\documentclass[a4paper,reqno,12pt]{amsart}
\usepackage[all, 2cell, knot,]{xy} \UseAllTwocells \SilentMatrices
\usepackage{xspace}
\usepackage{amsmath}
\usepackage{amstext}
\usepackage{amsfonts}
\usepackage[mathscr]{euscript}
\usepackage{amscd}
\usepackage{latexsym}
\usepackage{amssymb}
\usepackage{layout}
\usepackage{amsthm}
\usepackage{amsfonts}
\usepackage{amsxtra}
\usepackage{color}
\usepackage[usenames,dvipsnames,svgnames,table]{xcolor}
\usepackage{graphics}
\usepackage{bm}
\usepackage{tabulary}
\usepackage{fancyhdr}
\usepackage[bookmarks=true,hyperfootnotes=false]{hyperref}
\hypersetup{
			colorlinks=true,
			linkcolor=blue,
			anchorcolor=black,
			citecolor=green,
			urlcolor=black,
}
\theoremstyle{plain}% default, italic fonts for text of the following items
\newtheorem{theorem}{Theorem}[section]
\newtheorem{lemma}[theorem]{Lemma}
\newtheorem{proposition}[theorem]{Proposition}
\newtheorem{corollary}[theorem]{Corollary}
\theoremstyle{definition}% roman font for text of the following items
\newtheorem{definition}[theorem]{Definition}

\newtheorem{notation}[theorem]{Notation}
\newtheorem{remark}[theorem]{Remark}

\newcommand{\clC}{\mathcal{C}}
\newcommand{\clD}{\mathcal{D}}
\newcommand{\clE}{\mathcal{E}}

\newcommand{\clI}{\mathcal{I}}

\newcommand{\zg}{\gamma}

\newcommand{\zd}{\delta}
\newcommand{\zD}{\Delta}

\newcommand{\pt}{\partial}

\newcommand{\Id}{\operatorname{Id}}

\newcommand{\pr}{\operatorname{pr}}
\newcommand{\nequ}{\mbox{$n$-equivalence}}
\newcommand{\equ}[1]{\mbox{$(#1)$-equivalence}}

\newcommand{\nid}{\noindent}

\newcommand{\bk}{\bigskip}
\newcommand{\mk}{\medskip}

\newcommand{\ovl}[1]{\overline{#1}}
\newcommand{\up}[1]{^{(#1)}}
\newcommand{\lo}[1]{_{(#1)}}
\newcommand{\rw}{\rightarrow}
\newcommand{\Rw}{\Rightarrow}

\newcommand{\xrw}{\xrightarrow} % use as follows: \xrw{} brackets possibly empty
\newcommand{\tiund}[1]{{\times}_{#1}\:}
\newcommand{\pro}[3]{#1\tiund{#2}\overset{#3}{\cdots}\tiund{#2}#1}
\newcommand{\tens}[2]{#1\,\tiund{#2}\,#1}
\newcommand{\uset}[2]{\underset{#1}{#2}}
\newcommand{\oset}[2]{\overset{#1}{#2}}

\newcommand{\mi}{\text{-}}
\newcommand{\nm}{(n-1)}
\newcommand{\bl}{\bullet}
\newcommand{\cop}{\textstyle{\,\coprod\,}}

\newcommand{\seqc}[3]{{#1}_{#2},...,{#1}_{#3}}

\newcommand{\dop}[1]{\Delta^{{#1}^{op}}}
\newcommand{\Dop}{\Delta^{op}}
\newcommand{\Dnop}{\Delta^{{n}^{op}}}
\newcommand{\Dmenop}{\Delta^{{n-1}^{op}}}
\newcommand{\cat}[1]{\mbox{$\mathsf{Cat^{#1}}$}}
\newcommand{\Cat}{\mbox{$\mathsf{Cat}\,$}}
\newcommand{\Catn}{\mbox{$\mathsf{Cat^n}$}}
\newcommand{\Gpd}{\mbox{$\mathsf{Gpd}$}}

\newcommand{\cathd}[1]{\mbox{$\mathsf{Cat_{hd}^{#1}}$}}

\newcommand{\catwg}[1]{\mbox{$\mathsf{Cat_{wg}^{#1}}$}}

\newcommand{\segpsc}[2]{\mbox{$\mathsf{SegPs}$}\funcat{#1}{#2}}
\newcommand{\Ps}{\mbox{$\mathsf{Ps}$}}
\newcommand{\psc}[2]{\mbox{$\mathsf{Ps}$}\funcat{#1}{#2}}
\newcommand{\muk}{\mu_k}
\newcommand{\hmu}[1]{\hat\mu_{#1}}
\newcommand{\hmuk}{\hat{\mu}_k}

\newcommand{\Nn}{N_{(n)}}
\newcommand{\N}[1]{N_{(#1)}}
\newcommand{\Nu}[1]{N^{(#1)}}
\newcommand{\Nb}[1]{N_{(#1)}}
\newcommand{\zgu}[1]{\zg^{(#1)}}
\newcommand{\zgb}[1]{\zg_{(#1)}}
\newcommand{\di}[1]{d^{(#1)}}

\newcommand{\p}[1]{p^{(#1)}}

\newcommand{\pn}{p^{(n)}}
\newcommand{\op}[1]{\bar{p}^{(#1)}}

\newcommand{\Set}{\mbox{$\mathsf{Set}$}}
\newcommand{\St}{St\,}

\newcommand{\uk}{\underline{k}}
\newcommand{\ur}{\underline{r}}

\newcommand{\funcat}[2]{[\Delta^{{#1}^{op}},#2]}

\newcommand{\nfol}{$n$-fold }
\begin{document}

\title [\tiny{Pseudo-functors modelling higher structures}]{Pseudo-functors modelling higher structures\\ }

\author[\tiny{Simona Paoli}]{Simona Paoli}
 \address{Department of Mathematics, University of Leicester,
LE17RH, UK}
 \email{sp424@le.ac.uk}

\date{22 May 2016}



%%%%%%%%%%%%%%%%%%%%%%%%%%%%%%%%%%%%%%%%%%%%%%%%%%%%%%%%%%%%%%%%%%%%%%%%%%%%%

\begin{abstract}
We introduce a new higher categorical structure called a weakly globular \nfol category. This structure is based on iterated internal categories and on the notion of weak globularity. We identify a suitable class of pseudo-functors whose strictification produces weakly globular \nfol categories.
\end{abstract}

\maketitle
%\thispagestyle{fancy}

%\tableofcontents

%%%%%%%%%%%%%%%%%%%%%%%%%%%%%%%%%%%%%%%%%%%%%%%%%%%%%%%%%%%%%%%%%%%%%%%%%%%%%%%
\section{Introduction}\label{sec-intro}
In this paper we introduce a new higher categorical structure called a weakly globular \nfol category. We show that this arises from pseudo-functors, a widely used notion in category theory and homotopy  theory.

Weakly globular \nfol categories are a model of higher categories based on multi-simplicial sets. Simplicial and multi-simplicial models of higher categories have been developed over the years in several different contexts.

In the ground-breaking work of Lurie \cite{L2} and Joyal \cite{Jo} on quasi-categories, simplicial sets with additional properties lead to $(\infty,1)$-categories; that is, higher categories in which the cells in dimension higher than 1 are invertible. Several Quillen equivalent model structures describe $(\infty,1)$-categories, including simplicial categories \cite{Be2}, Segal categories \cite{Be2}, complete Segal spaces \cite{Re2}, relative categories \cite{Bk1}. More recently, the notion of $(\infty,n)$-categories emerged \cite{Re1} \cite{BeRe}, modelling higher categories where the cells in dimension higher than $n$ are invertible.
Several Quillen model structures on $(\infty,n)$-categories  have been developed, see \cite{Be3}. Simplicial methods also feature prominently in describing weak $\omega$-categories via  complicial sets \cite{Ve}.

In this work we study higher categories in the truncated $n$-case. These structures generalize  categories because in addition to objects and arrows they admit higher arrows (also called higher cells) and composition between them. When compositions are  associative and unital, we obtain  strict $n$-categories. The latter are insufficient for  many applications of higher category theory. For instance, strict $n$-groupoids do not model $n$-types in dimension $n>2$ (see \cite{S2} for a counterexample in dimension $n=3$).

The wider class of weak $n$-categories is needed: in a weak $n$-category, compositions are associative and unital only up to an invertible cell in the next dimension and these associativity and unit isomorphisms are suitably compatible or coherent.

In dimension $n=2$ and $n=3$ the idea of weak $n$-category is embodied in the classical notions of bicategory \cite{Ben} and tricategory \cite{GPS}. In these structures, explicit diagrams encode the coherence axioms for the associativity and unit isomorphisms. Capturing the coherence axioms explicitly in dimension $n>3$ seems untractable. Instead, different combinatorial machineries to define weak $n$-category \cite{L1} emerged: in these models the coherence data for the higher associativities are not given explicitly but they are automatically encoded in the combinatorics defining the models.

Different types of combinatorics have been used, including multi-simplicial sets as in Tamsamani and Simpson \cite{S2}, \cite{Ta}, (higher) operads as in Batanin  \cite{B}, Leinster \cite{L1} and Trimble  \cite{Cheng2}, opetopes as in \cite{BD2}, \cite{Cheng1} and several others.

In these classical models of higher categories the cells in dimension 0 up to $n$ form a discrete structure, that is a set. We call this the \emph{globularity condition} since it determines to the globular shape of the higher cells in the structure.

Weakly globular \nfol categories are based on a new paradigm to weaken higher categorical structure: the idea of \emph{weak globularity}. In our approach the cells in dimension 0 up to $n$ no longer form a set but have a higher categorical structure suitably equivalent to a discrete one. More precisely, they form a 'homotopically discrete \nfol category' as defined by the author in \cite{Pa}. We call this the \emph{weak globularity condition}.

In subsequent papers \cite{Pa3} \cite{Pa4} we show that weakly globular \nfol categories are suitably equivalent to the Tamsamani-Simpson model of weak $n$-categories \cite{Simp} \cite{Ta}.

 Weakly globular \nfol categories form a full subcategory of \nfol categories. These are, inductively, internal categories in $(n-1)$-fold categories. An \nfol category is therefore a 'rigid' structure in which all compositions are associative and unital.

 The weakness in a weakly globular \nfol category is encoded by the weak globularity condition. The discretizations of the homotopically discrete structures in a weakly globular $n$-fold category play the role of sets of cells in the respective dimensions. Additional conditions are imposed in the definition of weakly globular \nfol category to obtain well behaved compositions of higher cells.

In the case $n=2$, weakly globular double categories were introduced in joint work by the author in \cite{BP} and shown to be biequivalent to bicategories. The generalization to the case $n>2$ is much more complex.

In previous work the author developed the notion of weakly globular \nfol structure in all dimension $n$ in a homotopical context: for the modeling of connected $n$-types with weakly globular cat$^{n}$-groups \cite{Pa}, and for the modeling of general $n$-types with weakly globular $n$-fold groupoids \cite{BP}.

 This paper stretches far beyond a categorical generalization of the higher groupoidal case. In this paper we connect our new structure to pseudo-functors from a small category into the $2$-category $\Cat$. In subsequent work \cite{Pa3}, \cite{Pa4} this will lead to the proof of a suitable equivalence between weakly globular $n$-fold categories and the Tamsamani-Simpson model.

 Pseudo-functors feature prominently in homotopy theory, for instance in iterated loop space theory \cite{Thomas}. They are also ubiquitous in category theory \cite{Borc}, and can be described with the language of $2$-monad and their pseudo-algebras \cite{PW}.

  Weakly globular $n$-fold categories are a full subcategory of $(n-1)$-fold simplicial objects in $\Cat$, that is functors $\funcat{n-1}{\Cat}$. We consider the pseudo-version of these, that is pseudo-functors $\psc{n-1}{\Cat}$.

 Crucial to this work is the use of the strictification of pseudo-functors into strict functors. This topic had many contributions in category theory, including \cite{Str}. We use in this work the elegant formulation of Power \cite{PW}, further refined by Lack in \cite{Lack}. The latter were recently generalized in \cite{Hai}, but for this work \cite{PW} and \cite{Lack} are sufficient.

  The classical theory of strictification of pseudo-algebras \cite{PW}, \cite{Lack} affords the strictification functor

 \begin{equation}\label{eq1-intro}
    \St:\psc{n-1}{\Cat} \rw \funcat{n-1}{\Cat}
\end{equation}
left adjoint to the inclusion.

The coherence axioms in a pseudo-functor are reminiscent of the coherence data for the compositions of higher cells in a weak higher category. So it is natural to ask if a subcategory of pseudo-functors can model, in a suitable sense, higher structures. In this paper we positively answer this question by introducing a subcategory

\begin{equation*}
  \segpsc{n-1}{\Cat}\subset \psc{n-1}{\Cat}
\end{equation*}
of \emph{Segalic pseudo-functors}. Our main result, Theorem \ref{the-strict-funct} is that the classical strictification functor \eqref{eq1-intro} restricts to a functor
\begin{equation*}
    L_n:\segpsc{n-1}{\Cat} \rw\catwg{n}\;.
\end{equation*}
In a subsequent paper \cite{Pa3} we associate to a Tamsamani weak $n$-category a Segalic pseudo-functor and thus build a 'rigidification' functor from Tamsamani weak $n$-categories to weakly globular \nfol categories.

This paper is organized as follows. Section \ref{sec-prelim} contains some preliminaries on (multi) simplicial techniques as well as on pseudo-functors and their strictification.

Section \ref{sec-wg-nfold-categ} introduces weakly globular \nfol categories and $n$-equi-valences between them, and discusses the main properties of this structure. In Proposition \ref{pro-crit-ncat-be-wg} we prove a criterion for a \nfol category to be weakly globular. This is crucial to prove the main result in the next section.

Section \ref{sec-pseud-strict} introduces Segalic pseudo-functors. We show in Proposition \ref{pro-transf-wg-struc} that an \nfol category levelwise equivalent to a Segalic pseudo-functor is weakly globular. This leads to the main result Theorem \ref{the-strict-funct}.

\bk

\textbf{Acknowledgements}: This work has been supported by a Marie Curie International Reintegration Grant No 256341. I thank the Centre for Australian Category Theory for their hospitality and financial support during August-December 2015, as well as the University of Leicester for its financial support during my study leave. I also thank the University of Chicago for their hospitality and financial support during April 2016.

%%%%

\section{Preliminaries}\label{sec-prelim}
In this section we review some basic simplicial techniques that we will use throughout the paper as well as some categorical background on pseudo-functors and their strictification. The material in this section is well-known, see for instance \cite{Borc}, \cite{Jard}, \cite{PW}, \cite{Lack}.
\subsection{Simplicial objects}\label{sbs-simp-tech}
Let $\zD$ be the simplicial category and let $\Dnop$ denote the product of $n$ copies of $\Dop$. Given a category $\clC$, $\funcat{n}{\clC}$ is called the category of $n$-simplicial objects in $\clC$ (simplicial objects in $\clC$ when $n=1$).
\begin{notation}\label{not-simp}
    If $X\in \funcat{n}{\clC}$ and $\uk=([k_1],\ldots,[k_n])\in \Dnop$, we shall denote $X ([k_1],\ldots,[k_n])$ by $X(k_1,\ldots,k_n)$, as well as $X_{k_1,\ldots,k_n}$ and $X_{\uk}$. We shall also denote $\uk(1,i)=([k_1],\ldots,[k_{i-1}],1,[k_{i+1}],\ldots,[k_n]) \in \Dnop$ for $1\leq i\leq n$.

    Every $n$-simplicial object in $\clC$ can be regarded as a simplicial object in $\funcat{n-1}{\clC}$ in $n$ possible ways. For each $1\leq i\leq n$ there is an isomorphism
    \begin{equation*}
        \xi_i:\funcat{n}{\clC}\rw\funcat{}{\funcat{n-1}{\clC}}
    \end{equation*}
    given by
    \begin{equation*}
        (\xi_i X)_{r}(k_1,\ldots,k_{n-1})=X(k_1,\ldots,k_{i-1},r,k_{i+1},\ldots,k_{n-1})
    \end{equation*}
    for $X\in\funcat{n}\clC$ and $r\in\Dop$.
\end{notation}
\begin{definition}\label{def-fun-smacat}
    Let $F:\clC \rw \clD$ be a functor, $\clI$ a small category. Denote
    \begin{equation*}
        \ovl{F}:[\clI,\clC]\rw [\clI,\clD]
    \end{equation*}
    the functor given by
    \begin{equation*}
        (\ovl{F} X)_i = F(X(i))
    \end{equation*}
    for all $i\in\clI$.
\end{definition}
\begin{definition}\label{def-seg-map}
    Let ${X\in\funcat{}{\clC}}$ be a simplicial object in any category $\clC$ with pullbacks. For each ${1\leq j\leq k}$ and $k\geq 2$, let ${\nu_j:X_k\rw X_1}$ be induced by the map  $[1]\rw[k]$ in $\Delta$ sending $0$ to ${j-1}$ and $1$ to $j$. Then the following diagram commutes:
\begin{equation}\label{eq-seg-map}
\xymatrix{
&&&& X\sb{k} \ar[llld]_{\nu\sb{1}} \ar[ld]_{\nu\sb{2}} \ar[rrd]^{\nu\sb{k}} &&&& \\
& X\sb{1} \ar[ld]_{d\sb{1}} \ar[rd]^{d\sb{0}} &&
X\sb{1} \ar[ld]_{d\sb{1}} \ar[rd]^{d\sb{0}} && \dotsc &
X\sb{1} \ar[ld]_{d\sb{1}} \ar[rd]^{d\sb{0}} & \\
X\sb{0} && X\sb{0} && X\sb{0} &\dotsc X\sb{0} && X\sb{0}
}
\end{equation}

If  ${\pro{X_1}{X_0}{k}}$ denotes the limit of the lower part of the
diagram \eqref{eq-seg-map}, the \emph{$k$-th Segal map for $X$} is the unique map
$$
\muk:X\sb{k}~\rw~\pro{X\sb{1}}{X\sb{0}}{k}
$$
\noindent such that ${\pr_j\,\muk=\nu\sb{j}}$ where
${\pr\sb{j}}$ is the $j^{th}$ projection.
\end{definition}
\begin{definition}\label{def-ind-seg-map}

    Let ${X\in\funcat{}{\clC}}$ and suppose that there is a map $\zg: X_0 \rw X^d_0$ in $\clC$  $\zg: X_0 \rw X^d_0$  such that the limit of the diagram
\begin{equation*}
\xymatrix@C20pt{
& X\sb{1} \ar[ld]_{\zg d\sb{1}} \ar[rd]^{\zg d\sb{0}} &&
X\sb{1} \ar[ld]_{\zg d\sb{1}} \ar[rd]^{\zg d\sb{0}} &\cdots& k &\cdots&
X\sb{1} \ar[ld]_{\zg d\sb{1}} \ar[rd]^{\zg d\sb{0}} & \\
X^d\sb{0} && X^d\sb{0} && X^d\sb{0}\cdots &&\cdots X^d\sb{0} && X^d\sb{0}
    }
\end{equation*}
exists; denote the latter by $\pro{X_1}{X_0^d}{k}$. Then the following diagram commutes, where $\nu_j$ is as in Definition \ref{def-seg-map}, and $k\geq 2$
\begin{equation*}
\xymatrix{
&&&& X\sb{k} \ar[llld]_{\nu\sb{1}} \ar[ld]_{\nu\sb{2}} \ar[rrd]^{\nu\sb{k}} &&&& \\
& X\sb{1} \ar[ld]_{\zg d\sb{1}} \ar[rd]^{\zg d\sb{0}} &&
X\sb{1} \ar[ld]_{\zg d\sb{1}} \ar[rd]^{\zg d\sb{0}} && \dotsc &
X\sb{1} \ar[ld]_{\zg d\sb{1}} \ar[rd]^{\zg d\sb{0}} & \\
X^d\sb{0} && X^d\sb{0} && X^d\sb{0} &\dotsc X^d\sb{0} && X^d\sb{0}
}
\end{equation*}
The \emph{$k$-th induced Segal map for $X$} is the unique map
\begin{equation*}
\hmuk:X\sb{k}~\rw~\pro{X\sb{1}}{X^d\sb{0}}{k}
\end{equation*}
such that ${\pr_j\,\hmuk=\nu\sb{j}}$ where ${\pr\sb{j}}$ is the $j^{th}$ projection.
\end{definition}
\subsection{$\mathbf{n}$-Fold internal categories}\label{sbs-nint-cat}

Let  $\clC$ be a category with finite limits. An internal category $X$ in $\clC$ is a diagram in $\clC$
\begin{equation}\label{eq-nint-cat}
\xymatrix{
\tens{X_1}{X_0} \ar^(0.65){m}[r] & X_1 \ar^{d_0}[r]<2.5ex> \ar^{d_1}[r] & X_0
\ar^{s}[l]<2ex>
}
\end{equation}
where $m,d_0,d_1,s$ satisfy the usual axioms of a category (see for instance \cite{Borc}) for details. An internal functor is a morphism of diagrams like \eqref{eq-nint-cat} commuting in the obvious way. We denote by $\Cat \clC$ the category of internal categories and internal functors.

The category $\cat{n}(\clC)$ of \nfol categories in $\clC$ is defined inductively by iterating $n$ times the internal category construction. That is, $\cat{1}(\clC)=\Cat$ and, for $n>1$,
\begin{equation*}
  \cat{n}(\clC)= \Cat(\cat{n-1}(\clC)).
\end{equation*}

When $\clC=\Set$, $\cat{n}(\Set)$ is simply denoted by $\cat{n}$ and called the category of \nfol categories (double categories when $n=2$).

%%%%%%%%%%%%%%%%%%%%%%%%%%%%%%%%%%%%%%%%%%%%%%%%%%%%%%%%%%%%%%%%%%%%%%%%%%%%
\subsection{Nerve functors}\label{sus-ner-funct}

There is a nerve functor
\begin{equation*}
    N:\Cat\clC \rw \funcat{}{\clC}
\end{equation*}
such that, for $X\in\Cat\clC$
\begin{equation*}
    (N X)_k=
    \left\{
      \begin{array}{ll}
        X_0, & \hbox{$k=0$;} \\
        X_1, & \hbox{$k=1$;} \\
        \pro{X_1}{X_0}{k}, & \hbox{$k>1$.}
      \end{array}
    \right.
\end{equation*}
When no ambiguity arises, we shall sometimes denote $(NX)_k$ by $X_k$ for all $k\geq 0$.

The following fact is well known:
\begin{proposition}\label{pro-ner-int-cat}
    A simplicial object in $\clC$ is the nerve of an internal category in $\clC$ if and only if all the Segal maps are isomorphisms.
\end{proposition}

By iterating the nerve construction, we obtain the multinerve functor
\begin{equation*}
    \Nn :\cat{n}(\clC)\rw \funcat{n}{\clC}\;.
\end{equation*}
\begin{definition}\label{def-discrete-nfold}
An internal $n$-fold category $X\in \cat{n}(\clC)$ is said to be discrete if $\Nn X$ is a constant functor.
\end{definition}

 Each object of $\cat{n}(\clC)$ can be considered as an internal category in $\cat{n-1}(\clC)$ in $n$ possible ways, corresponding to the $n$ simplicial directions of its multinerve. To prove this, we use the following lemma, which is a straightforward consequence of the definitions.

\newpage
\begin{lemma}\label{lem-multin-iff}\
\begin{itemize}
      \item [a)] $X\in\funcat{n}{\clC}$ is the multinerve of an \nfol category in $\clC$ if and only if, for all $1\leq r\leq n$ and $[p_1],\ldots,[p_r]\in\Dop$, $p_r\geq 2$
\begin{equation}\label{eq-multin-iff}
\begin{split}
    &  X(p_1,\ldots,p_r,\mi)\cong\\
    &\cong\pro{X(p_1,\ldots,p_{r-1},1,\mi)}{X(p_1,\ldots,p_{r-1},0,\mi)}{p_r}
\end{split}
\end{equation}
      \item [b)] Let $X\in\cat{n}(\clC)$. For each $1\leq k\leq n$, $[i]\in\Dop$, there is $X_i\up{k}\in\cat{n-1}(\clC)$ with
\begin{equation*}
    \Nb{n-1}X_i\up{k} (p_1,\ldots,p_{n-1})=\Nn X(p_1,\ldots,p_{k-1},i,p_k,\ldots,p_{n-1})
\end{equation*}
    \end{itemize}
\end{lemma}
\begin{proof}\

  a) By induction on $n$. By Proposition \ref{pro-ner-int-cat}, it is true for $n=1$. Suppose it holds for $n-1$ and let $X\in\Cat(\cat{n-1}(\clC))$ with objects of objects (resp. arrows) $X_0$ (resp. $X_1$); denote $(NX)_p=X_p$. By definition of the multinerve
      \begin{equation*}
        (\Nb{n} X)(p_1,\ldots,p_r,\mi)=\Nb{n-1}X_{p_1}(p_2,\ldots,p_r,\mi)\;.
      \end{equation*}
      Hence using the induction hypothesis
\begin{flalign*}
       &\Nb{n}X(p_1...p_r\,\mi)=\Nb{n-1}X_{p_1}(p_2... p_r\,\mi)\cong&
\end{flalign*}
\begin{equation*}
\resizebox{1.0\hsize}{!}{$
\cong \pro{\Nb{n-1} X_{p_1}(p_2... p_{r-1}\,1\,\mi)}
         {\Nb{n-1} X_{p_1}(p_2... p_{r-1}\,0\,\mi)}{p_r}=$}
\end{equation*}
\begin{equation*}
\resizebox{.97\hsize}{!}{
          $=\pro{\Nb{n} X(p_1... p_{r-1}\,1\,\mi)}
         {\Nb{n} X(p_1... p_{r-1}\,0\,\mi)}{p_r}.$} \hspace{10mm}
\end{equation*}
Conversely, suppose $X\in\funcat{n}{\clC}$ satisfies \eqref{eq-multin-iff}. Then for each $[p]\in\Dop$, $X(p,\mi)$ satisfies \eqref{eq-multin-iff}, hence
\begin{equation*}
    X(p,\mi)=\Nb{n-1}X_p
\end{equation*}
for $X_p\in\cat{n-1}(\clC)$. Also, by induction hypothesis
\begin{equation*}
    X(p,\mi)=\pro{X(1,\mi)}{X(0,\mi)}{p}\;.
\end{equation*}
Thus we have the object $X\in\cat{n}(\clC)$ with objects $X_0$, arrows $X_1$ and $X_p=X(p,\mi)$ as above.

\bigskip
b) By part a), there is an isomorphism for $p_r\geq 2$
\begin{flalign*}
       &\Nb{n}X(p_1...p_n)=&
\end{flalign*}
\begin{equation*}
\resizebox{1.0\hsize}{!}{$\pro{\Nb{n}X(p_1...p_{r-1}\, 1 ...p_n)}{\Nb{n}X(p_1...p_{r-1}\, 0 ...p_n)}{p_r}$}\;.
\end{equation*}
In particular, evaluating this at $p_k=i$, this is saying the $(n-1)$-simplicial group taking $(p_1...p_n)$ to $\Nb{n}X(p_1...p_{k-1}\, i ...p_{n-1})$ satisfies condition \eqref{eq-multin-iff} in part a). Hence by part a) there exists $X_i\up{k}$ with
\begin{equation*}
    \Nb{n-1}X_i\up{k}(p_1...p_{n-1})=\Nb{n}X(p_1...p_{k-1}\, i ...p_{n-1})
\end{equation*}
as required.
\end{proof}
\begin{proposition}\label{pro-assoc-iso}
    For each $1\leq k\leq n$ there is an isomorphism $\xi_k:\cat{n}(\clC)\rw \cat{n}(\clC)$ which associates to $X=\cat{n}(\clC)$ an object $\xi_k X$ of $\Cat(\cat{n-1}(\clC))$ with
    \begin{equation*}
        (\xi_k X)_i=X_i\up{k}\qquad i=0,1
    \end{equation*}
    with $X_i\up{k}$ as in Lemma \ref{lem-multin-iff}.
\end{proposition}
\begin{proof}
Consider the object of $\funcat{}{\funcat{n-1}{\clC}}$ taking $i$ to the $\nm$-simplicial object associating to $(\seqc{p}{1}{n-1})$ the object
\begin{equation*}
    \Nb{n}X(p_1...p_{k-1}\, i \, p_{k+1}...p_{n-1})\;.
\end{equation*}
By Lemma \ref{lem-multin-iff} b), the latter is the multinerve of $X_i\up{k}\in \cat{n-1}(\clC)$. Further, by Lemma \ref{lem-multin-iff} a), we have
\begin{equation*}
    \Nb{n-1}X_i\up{k}\cong \pro{\Nb{n-1}X_1\up{k}}{\Nb{n-1}X_0\up{k}}{i}\;.
\end{equation*}
Hence $\Nb{n}X$ as a simplicial object in $\funcat{n-1}{\clC}$ along the $k^{th}$ direction, has
\begin{equation*}
    (\Nb{n}X)_i=
    \left\{
      \begin{array}{ll}
        \Nb{n-1}X_i\up{k}, & \hbox{$i=0,1$;} \\
        \Nb{n-1}(\pro{X_1\up{k}}{X_0\up{k}}{i}), & \hbox{for $i\geq 2$.}
      \end{array}
    \right.
\end{equation*}
This defines $\xi_k X\in\Cat(\cat{n-1}(\clC))$ with
\begin{equation*}
    (\xi_k X)_i=\Nb{n-1}X_i\up{k}\qquad i=0,1\;.
\end{equation*}
We now define the inverse for $\xi_k$. Let $X\in\Cat(\cat{n-1}(\clC))$, and let $X_i=\pro{X_1}{X_0}{i}$ for $i\geq 2$. The $n$-simplicial object $X_{k}$ taking $(p_1,\ldots,p_n)$ to
\begin{equation*}
    \Nb{n}X_{p_k}(p_1...p_{k-1}p_{k+1}...p_n)
\end{equation*}
satisfies condition \eqref{eq-multin-iff}, as easily seen. Hence by Lemma \ref{lem-multin-iff} there is $\xi'_k X\in\cat{n}\clC$ such that $\Nb{n }\xi'_k X=X_k$. It is immediate to check that $\xi_k$ and $\xi'_k$ are inverse bijections.
\end{proof}
\begin{definition}\label{def-ner-func-dirk}
    The nerve functor in the $k^{th}$ direction is defined as the composite
    \begin{equation*}
        \Nu{k}:\cat{n}(\clC)\xrw{\xi_k}\Cat(\cat{n-1}(\clC))\xrw{N}\funcat{}{\cat{n-1}(\clC)}
    \end{equation*}
    so that, in the above notation,
    \begin{equation*}
        (\Nu{k}X)_i=X\up{k}_i\qquad i=0,1\;.
    \end{equation*}
    Note that $\Nb{n}=\Nu{n}...\Nu{2}\Nu{1}$.
\end{definition}
\begin{notation}\label{not-ner-func-dirk}
    When $\clC=\Set$ we shall denote
    \begin{equation*}
        J_n=\Nu{n-1}\ldots \Nu{1}:\cat{n}\rw\funcat{n-1}{\Cat}\;.
    \end{equation*}
\end{notation}
Thus $J_n$ amounts to taking the nerve construction in all but the last simplicial direction. Thus $J_n$ amounts to taking the nerve construction in all but the last simplicial direction. The functor $J_n$ is fully faithful, thus we can identify $\cat{n}$ with the image $J_n(\cat{n})$ of the functor $J_n$.

 Given $X\in\cat{n}$, when no ambiguity arises we shall denote, for each $(s_1,\ldots ,s_{n-1})\in\Dmenop$
\begin{equation*}
    X_{s_1,\ldots ,s_{n-1}}=(J_n X)_{s_1,\ldots ,s_{n-1}}\in\Cat
\end{equation*}
and more generally, if $1\leq j \leq n-1$,
\begin{equation*}
    X_{s_1,\ldots ,s_{j}}=(\Nu{j}\ldots \Nu{1} X)_{s_1,\ldots ,s_{j}}\in\cat{n-j}\;.
\end{equation*}
Let $ob : \Cat \clC \rw \clC$ be the object of object functor. The left adjoint to $ob$ is the discrete internal category functor $d$. By Proposition \ref{pro-assoc-iso}  we then have
\begin{equation*}
\xymatrix{
    \cat{n}\clC \oset{\xi_n}{\cong}\Cat(\cat{n-1}\clC) \ar@<1ex>[r]^(0.65){ob} & \cat{n-1}\clC \ar@<1ex>[l]^(0.35){d}\;.
}
\end{equation*}

We denote
\begin{equation*}
\di{n}=\xi^{-1}_{n}\circ d \text{\;\;for\;\;} n>1,\; \di{1}=d \;.
\end{equation*}
Thus $\di{n}$ is the discrete inclusion of $\cat{n-1}\clC$ into $\cat{n}\clC$ in the $n^{th}$ direction.

\bk

The following is a characterization of objects of $\funcat{n-1}{\Cat}$ in the image of the functor $J_n$ in \ref{not-ner-func-dirk}.
\begin{lemma}\label{lem-char-obj}
     Let $L\in \funcat{n-1}{\Cat}$ be such that, for all $\uk\in\dop{n-1}$, $1\leq i\leq n-1$ and $k_i\geq 2$, the Segal maps are isomorphisms:
     \begin{equation}\label{eq-lem-char-obj}
        L_{\uk} \cong\pro{L_{\uk(1,i)}}{L_{\uk(0,i)}}{k_i}\;.
     \end{equation}
     Then $L\in\cat{n}$.
\end{lemma}
\begin{proof}
By induction on $n$. When $n=2$, $L\in\funcat{}{\Cat}$, $k\in \Dop$, $i=1$, $\uk(1,i)=1$, $\uk(0,i)=0$, $k_i=k=2$ and
\begin{equation*}
    L_k\cong\pro{L_1}{L_0}{k}\;.
\end{equation*}
Thus by Proposition \ref{pro-ner-int-cat}, $L\in\cat{2}$.

Suppose the lemma holds for $(n-1)$ and let $L \in \funcat{n-1}{\Cat}$ be as in the hypothesis. Consider $L_j\in\funcat{n-2}{\Cat}$ for $j\geq 0$. Let $\ur\in\dop{n-2}$ and denote $\uk=(j,\ur)\in\dop{n-1}$. Then, for any $2\leq i\leq n-1$, $k_i=r_{i-1}$ and
\begin{equation*}
    L_{\uk}=(L_j)_{\ur},\quad L_{\uk(1,i)}=(L_j)_{\ur(1,i-1)},\quad L_{\uk(0,i)}=(L_j)_{\ur(0,i-1)}\;.
\end{equation*}
Therefore \eqref{eq-lem-char-obj} implies
\begin{equation*}
    (L_j)_{\ur}=\pro{(L_j)_{\ur(1,i-1)}}{(L_j)_{\ur(0,i-1)}}{r_{i-1}}\;.
\end{equation*}
This means that $L_j$ satisfies the inductive hypothesis and therefore $L_j\in\cat{n-1}$.

Taking $i=1$ in \eqref{eq-lem-char-obj} we see that, for each $k_1\geq 2$ and $\ur=(k_2,\ldots,k_{n-1}\in\dop{n-2}$,
\begin{equation*}
    (L_{k_1})_{\ur}=(L_{k_1})_{\ur}=\pro{(L_1)_{\ur}}{(L_0)_{\ur}}{k_1}\;.
\end{equation*}
That is, we have isomorphisms in $\cat{n-1}$
\begin{equation*}
    L_{k_1}\cong \pro{L_1}{L_0}{k_1}\;.
\end{equation*}
We conclude from Proposition \ref{pro-ner-int-cat} the $L\in\cat{n}$.

\end{proof}

%%%%%%%%%%%%%%%%%%%%%%%%%%%%%%%%%%%%%%%%%%%%%%%%%%%%%%%%%%%%%%%%%%%%%%%
\subsection{Some functors on $ \Cat$}\label{sbs-funct-cat}

The connected component functor
\begin{equation*}
    q: \Cat\rw \Set
\end{equation*}
associates to a category its set of paths components. This is left adjoint to the discrete category functor
\begin{equation*}
    \di{1}:\Set \rw \Cat
\end{equation*}
associating to a set $X$ the discrete category on that set. We denote by
\begin{equation*}
    \zgu{1}:\Id\Rw \di{1}q
\end{equation*}
the unit of the adjunction $q\dashv \di{1}$.
\begin{lemma}\label{lem-q-pres-fib-pro}
    $q$ preserves fiber products over discrete objects and sends
    equivalences of categories to isomorphisms.
\end{lemma}
\begin{proof}
We claim that $q$ preserves products; that is, given categories
$\clC$ and $\clD$, there is a bijection
\begin{equation*}
    q(\clC\times \clD)=q(\clC)\times q(\clD)\;.
\end{equation*}
In fact, given $(c,d)\in q(\clC\times \clD)$ the map
$q(\clC\times\clD)\rw q(\clC)\times q(\clD)$ given by
$[(c,d)]=([c],[d])$ is well defined and is clearly surjective. On
the other hand, this map is also injective: given $[(c,d)]$ and
$[(c',d')]$ with $[c]=[c']$ and $[d]=[d']$, we have paths in $\clC$
\newcommand{\lin}{-\!\!\!-\!\!\!-}
\begin{equation*}
\xymatrix @R5pt{c \hspace{2mm} \lin \hspace{2mm}\cdots \hspace{2mm}
\lin
\hspace{2mm}c'\\
d \hspace{2mm} \lin \hspace{2mm}\cdots \hspace{2mm} \lin
\hspace{2mm}d' }
\end{equation*}
and hence a path in $\clC\times \clD$
\begin{equation*}
(c,d)\hspace{2mm}\lin\hspace{2mm}\cdots\hspace{2mm}\lin\hspace{2mm}(c',d)
\hspace{2mm}\lin\hspace{2mm}\cdots\hspace{2mm}\lin\hspace{2mm}(c',d')\;.
\end{equation*}
Thus $[(c,d)]=[(c',d')]$ and so the map is also injective, hence it
is a bijection, as claimed.

Given a diagram in $\Cat$ $\xymatrix{\clC\ar_{f}[r] & \clE & \clD
\ar^{g}[l]}$ with $\clE$ discrete, we have
\begin{equation}\label{eq-q-pres-fib-pro}
    \clC\tiund{\clE}\clD=\underset{x\in\clE}{\coprod}\clC_x\times
    \clD_x
\end{equation}
where $\clC_x,\;\clD_x$ are the full subcategories of $\clC$ and
$\clD$ with objects $c,\;d$ such that \;$f(c)=x=g(d)$. Since $q$ preserves
products and (being left adjoint) coproducts, we conclude by
\eqref{eq-q-pres-fib-pro} that
\begin{equation*}
    q(\clC\tiund{\clE}\clD)\cong q(\clC)\tiund{\clE}\,q(\clD)\;.
\end{equation*}
Finally, if $F:\clC\simeq \clD:G$ is an equivalence of categories,
$FG\,\clC\cong\clC$ and $FG\,\clD\cong \clD$ which implies that
$qF\,qG\,\clC\cong q\clC$ and $qF\,qG\,\clD\cong q\clD$, so $q\clC$
and $q\clD$ are isomorphic.
\end{proof}
The isomorphism classes of objects functot
\begin{equation*}
    p:\Cat\rw\Set
\end{equation*}
associates to a category the set of isomorphism classes of its objects. Notice that if $\clC$ is a groupoid, $p\clC=q\clC$.
\begin{lemma}\label{lem-p-pres-fib-pro}
    $p$ preserves pullbacks over discrete objects and sends
    equivalences of categories to isomorphisms.
\end{lemma}
\begin{proof}
For a category $\clC$, let $m\clC$ be its maximal subgroupoid. Then $p\clC=qm\clC$. Given a diagram in $\Cat$ $\xymatrix{\clC\ar_{f}[r] & \clE & \clD
\ar^{g}[l]}$ with $\clE$ discrete, we have
\begin{equation*}
    \clC\tiund{\clE}\clD=\underset{x\in\clE}{\coprod}\clC_x\times
    \clD_x\;.
\end{equation*}
Since, as easily seen, $m$ commutes with (co)products, and $m\clE=\clE$, we obtain $m(\clC\tiund{\clE}\clD)=m\clC\tiund{\clE}m\clD$; so by Lemma \ref{lem-q-pres-fib-pro},
\begin{align*}
    p(\clC\tiund{\clE}\clD)&=qm(\clC\tiund{\clE}\clD)=q(m\clC\tiund{\clE}m\clD)
    =qm\clC\tiund{q\clE}qm\clD=p\clC\tiund{\clE}p\clD\;.
\end{align*}
Finally, if $F:\clC\simeq \clD:G$ is an equivalence of categories, $FG\clC\cong \clC$ and $FG\clD\cong\clD$ which implies that $pF\,pG\,\clC\cong p\clC$ and $qF\,qG\,\clD\cong q\clD$, so $q\clC$ and $q\clD$ are isomorphic.
\end{proof}
\subsection{Pseudo-functors and their strictification}\label{sbs-pseudo-functors}
The functor 2-category $\funcat{n}{\Cat}$ is 2-monadic over $[ob(\Dnop),\Cat]$ where $ob(\Dnop)$ is the set of objects of $\Dnop$. Let
\begin{equation*}
    U:\funcat{n}{\Cat}\rw [ob(\Dnop),\Cat]
\end{equation*}
be the forgetful functor $(UX)_{\uk}=X_{\uk}$. Its left adjoint $F$ is given on objects by
\begin{equation*}
    (FH)_{\uk}=\underset{\ur\in ob(\Dnop)}{\coprod}\Dnop(\ur,\uk)\times H_{\ur}
\end{equation*}
for $H\in [ob(\Dmenop),\Cat]$, $\uk\in ob(\Dmenop)$. If $T$ is the monad corresponding to the adjunction $F\dashv U$, then
\begin{equation*}
    (TH)_{\uk}=\underset{\ur\in ob(\Dnop)}{\coprod}\Dnop(\ur,\uk)\times H_{\ur}
\end{equation*}
A pseudo $T$-algebra is given by $H\in [ob(\Dnop),\Cat]$,
\begin{equation*}
    h_{n}: \underset{\ur\in ob(\Dnop)}{\coprod}\Dnop(\ur,\uk)\times H_{\ur} \rw H_{\uk}
\end{equation*}
and additional data, as described in \cite{PW}. This amounts precisely to functors from $\Dnop$ to $\Cat$ and the 2-category $\sf{Ps\mi T\mi alg}$ of pseudo $T$-algebras corresponds to the 2-category $\Ps\funcat{n}{\Cat}$ of pseudo-functors, pseudo-natural transformations and modifications.

The strictification result proved in \cite{PW} yields that every pseudo-functor from $\Dnop$ to $\Cat$ is equivalent, in $\Ps\funcat{n}{\Cat}$, to a 2-functor.

Given a pseudo $T$-algebra as above, \cite{PW} consider the factorization  of $h:TH\rw H$ as
\begin{equation*}
    TH\xrw{v}L\xrw{g}H
\end{equation*}
with $v_{\uk}$ bijective on objects and $g_{\uk}$ fully faithful, for each $\uk\in\Dnop$. It is shown in \cite{PW} that it is possible to give a strict $T$-algebra structure $TL\rw L$ such that $(g,Tg)$ is an equivalence of pseudo $T$-algebras. It is immediate to see that, for each $\uk\in\Dnop$, $g_{\uk}$ is an equivalence of categories.

Further, it is shown in \cite{Lack} that $\St:\psc{n}{\Cat}\rw\funcat{n}{\Cat}$ as described above is left adjoint to the inclusion
\begin{equation*}
  J:\funcat{n}{\Cat}\rw\psc{n}{\Cat}
\end{equation*}
 and that the components of the units are equivalences in $\psc{n}{\Cat}$.

\section{Weakly globular n-fold categories}\label{sec-wg-nfold-categ}
In this section we define weakly globular \nfol categories and establish their main properties.

 The definition of weakly globular \nfol category uses the notion of homotopically discrete \nfol category from \cite{Pa1} in order to formulate the weak globularity condition. We recall this notion and its main properties in Section \ref{sbs-hom-disc-nfold}.

 In Section \ref{sbs-def-wg-nf-cat} we inductively define of weakly globular \nfol categories and of $n$-equivalences between them. In Section \ref{sbs-prop-wg-nf-cat} we establish the main properties of weakly globular \nfol categories. We show in Proposition \ref{pro-crit-ncat-be-wg} b) a criterion for a \nfol category to be weakly weakly globular playing a crucial role in the proof of the main result Theorem \ref{the-strict-funct}.
\subsection{Homotopically discrete \nfol categories}\label{sbs-hom-disc-nfold}
\begin{definition}\label{def-hom-dis-ncat}
    Define inductively the full subcategory $\cathd{n}\subset\cat{n}$ of homotopically discrete \nfol categories.

    For $n=1$, $\cathd{1}=\cathd{}$ is the category of  equivalence relations. Denote by $\p{1}=p:\Cat\rw\Set$ the isomorphism classes of object functor.

    Suppose, inductively, that for each $1\leq k\leq n-1$ we defined $\cathd{k}\subset\cat{k}$ and $k$-equivalences such that the following holds:
    \begin{itemize}
      \item [a)] The $k^{th}$ direction in $\cathd{k}$ is groupoidal; that is, if $X\in\cathd{k}$, $\xi_{k}X\in\Gpd(\cat{k-1})$ (where $\xi_{k}X$ is as in Proposition \ref{pro-assoc-iso}).

      \item [b)] There is a functor $\p{k}:\cathd{k}\rw\cathd{k-1}$ making the following diagram commute:
          \begin{equation}\label{eq-p-fun-def}
            \xymatrix{
            \cathd{k} \ar^{\Nu{k-1}...\Nu{1}}[rrr] \ar_{p^{(k)}}[d] &&& \funcat{k-1}{\Cat} \ar^{\bar p}[d]\\
            \cathd{k-1} \ar_{\N{k-1}}[rrr] &&& \funcat{k-1}{\Set}
            }
          \end{equation}
          Note that this implies that $(\p{k}X)_{s_1 ... s_{k-1}}=p X_{s_1 ... s_{k-1}}$ for all $(s_1 ... s_{k-1})\in\dop{k-1}$.

    \end{itemize}

   $\cathd{n}$ is the full subcategory of $\funcat{}{\cathd{n-1}}$ whose objects $X$ are such that
    \bigskip
    \begin{itemize}
      \item [(i)]  $\hspace{30mm} X_s\cong\pro{X_1}{X_0}{s} \quad \mbox{for all} \; s \geq 2.$

\medskip
    In particular this implies that $X\in \Cat(\Gpd(\cat{n-2})) =\Gpd(\cat{n-1})$ and the $n^{th}$ direction in $X$ is groupoidal.
    \medskip
      \item [(ii)] The functor
      \begin{equation*}
        \op{n-1}:\cathd{n}\subset \funcat{}{\cathd{n-1}}\rw\funcat{}{\cathd{n-2}}
      \end{equation*}
      restricts to a functor
      \begin{equation*}
        \p{n}:\cathd{n}\rw\cathd{n-1}
      \end{equation*}
     Note that this implies that $(\p{n}X)_{s_1 ... s_{n-1}}=p X_{s_1 ... s_{n-1}}$ and that the following diagram commutes
          \begin{equation}\label{eq-p-fun-def}
            \xymatrix{
            \cathd{n} \ar^{\Nu{n-1}...\Nu{1}}[rrr] \ar_{p^{(n)}}[d] &&& \funcat{n-1}{\Cat} \ar^{\bar p}[d]\\
            \cathd{n-1} \ar_{\N{n-1}}[rrr] &&& \funcat{n-1}{\Set}
            }
          \end{equation}
     \end{itemize}
\end{definition}
\mk
\begin{definition}\label{def-hom-dis-ncat-1}

    Denote by $\zgu{n}_X:X\rw \di{n}\p{n}X$ the morphism given by
    \begin{equation*}
        (\zgu{n}_X)_{s_1...s_{n-1}} :X_{s_1...s_{n-1}} \rw d p X_{s_1...s_{n-1}}
    \end{equation*}
    for all  $(s_1,...,s_{n-1})\in \dop{n-1}$. Denote by
    \begin{equation*}
        X^d =\di{n}\di{n-1}...\di{1}\p{1}\p{2}...\p{n}X
    \end{equation*}
    and by $\zg\lo{n}$ the composite
    \begin{equation*}
        X\xrw{\zgu{n}}\di{n}\p{n}X \xrw{\di{n}\zgu{n-1}} \di{n}\di{n-1}\p{n-1}\p{n}X \rw \cdots \rw X^d\;.
    \end{equation*}
    For each $a,b\in X_0^d$ denote by $X(a,b)$ the fiber at $(a,b)$ of the map
    \begin{equation*}
        X_1 \xrw{(d_0,d_1)} X_0\times X_0 \xrw{\zg\lo{n}\times\zg\lo{n}} X_0^d\times X_0^d\;.
    \end{equation*}
\end{definition}
\begin{definition}\label{def-hom-dis-ncat-1}
Define inductively $n$-equivalences in $\cathd{n}$. For $n=1$, a 1-equivalence is an equivalence of categories. Suppose we defined $\nm$-equivalences in $\cathd{n-1}$. Then a map $f:X\rw Y$ in $\cathd{n}$ is an $n$-equivalence if, for all $a,b \in X_0^d$, $f(a,b):X(a,b) \rw Y(fa,fb)$ and $\p{n}f$ are $\nm$-equivalences.
\end{definition}
The main properties of the category $\cathd{n}$ are summarized in the proposition below, whose proof can be found in \cite{Pa1}
\begin{proposition}\rm{ \cite{Pa1}}\label{pro-Pa1}
\begin{itemize}
  \item [a)] $\zgu{n}$ and $\zgb{n}$ are $(n-1)$-equivalences.\mk

  \item [b)] For each $X\in\cathd{n}$ the induced Segal maps
  \begin{equation*}
    X_k \rw \pro{X_1}{X_0^d}{k}
  \end{equation*}
  for each $k\geq 2$ are $(n-1)$-equivalences.\mk

  \item [c)] $f:X\rw Y$ in $\cathd{n}$ is an $n$-equivalence if and only if $X^d \cong Y^d$.\mk

  \item [d)] The functor $J_n:\cat{n}\rw \funcat{n-1}{\Cat}$ restricts to a functor $J_n: \cathd{n}\rw \funcat{n-1}{\cathd{}}$.
\end{itemize}
\end{proposition}
\subsection{The definition of weakly globular \nfol categories}\label{sbs-def-wg-nf-cat}
In this section we define the category $\catwg{n}$ of weakly globular \nfol categories and $n$-equivalences.

 The idea of the definition is to build the structure by induction on dimension starting with the category $\Cat$ with equivalences of categories.

  At dimension $n$, the structure is a full subcategory of simplicial objects in $\catwg{n-1}$. Unraveling this definition, this affords an embedding
\begin{equation*}
        \xymatrix{
        \catwg{n}\ar@{^(->}^(0.35){J_n}[r] & \funcat{n-1}{\Cat}\;.
        }
\end{equation*}
The first condition for $X\in\funcat{}{\catwg{n-1}}$ to be an object of $\catwg{n}$ is the weak globularity condition that $X_0$ is homotopically discrete.

 The set underlying the discrete $(n-1)$-fold category $X_0^d$ plays the role of set of cells in dimension $(n-1)$. When $0\leq r\leq n-1$, the set underlying $(J_n X)^d_{1\oset{r}{\cdots}10}$ corresponds to the set of $r$-cells.

    The next condition in the definition of $\catwg{n}$ is that the Segal maps
    \begin{equation*}
        X_k\rw \pro{X_1}{X_0}{k}
    \end{equation*}
are isomorphisms for all $k\geq 2$. Since each $X_k\in\cat{n-1}$, by the characterization of internal categories via the Segal condition (Proposition \ref{pro-ner-int-cat}) it follows that $X$ is an \nfol category.

We further require the induced Segal map condition stating that, for each $k\geq 2$, the maps in $\catwg{n-1}$
\begin{equation*}
    X_k \rw \pro{X_1}{X_0^d}{k}
\end{equation*}
are $(n-1)$-equivalences. This condition controls the compositions of higher cells and is the analogue of the Segal condition in the Tamsamani-Simpson model \cite{Ta}, \cite{Simp}.

 We finally require the existence of a truncation functor $\p{n}$ from $\catwg{n}$ to $\catwg{n-1}$ obtained by applying dimensionwise the isomorphism classes of object functor to the corresponding diagram in $\funcat{n-1}{\Cat}$. In the case $n=2$, this last condition is redundant.

The functor $\p{n}$ is used to define $n$-equivalences, thus completing the inductive step in the definition of $\catwg{n}$. The definition of $n$-equivalences is given in terms of two conditions: the first is a higher dimensional generalization of the notion of fully faithfulness of a functor, the second is a generalization of 'essentially surjective on objects'.

\begin{definition}\label{def-n-equiv}
    For $n=1$, $\catwg{1}=\Cat$ and $1$-equivalences are equivalences of categories.

    Suppose, inductively, that we defined $\catwg{n-1}$ and $(n-1)$-equivalences. Then $\catwg{n}$ is the full subcategory of $\funcat{}{\catwg{n-1}}$ whose objects $X$ are such that
    \begin{itemize}
      \item [a)] \textsl{Weak globularity condition} $X_0\in\cathd{n-1}$.\mk
      \item [b)] \textsl{Segal condition} For all $k\geq 2$ the Segal maps are isomorphisms:
      \begin{equation*}
        X_k\cong\pro{X_1}{X_0}{k}\;.
      \end{equation*}

      \item [c)] \textsl{Induced Segal condition} For all $k\geq 2$ the induced Segal maps
      \begin{equation*}
        X_k\rw\pro{X_1}{X^d_0}{k}
      \end{equation*}
      (induced by the map $\zg:X_0\rw X_0^d$) are $(n-1)$-equivalences.\mk

      \item [d)] \textsl{Truncation functor} There is a functor $\p{n}:\cathd{n}\rw\cathd{n-1}$ making the following diagram commute
      \begin{equation*}
        \xymatrix{
        \catwg{n} \ar^{J_n}[rr] \ar_{\p{n}}[d] && \funcat{n-1}{\Cat} \ar^{\ovl p}[d]\\
        \catwg{n-1} \ar^{\Nb{n-1}}[rr]  && \funcat{n-1}{\Set}
        }
    \end{equation*}
    \end{itemize}
    Given $a,b\in X_0^d$, denote by $X(a,b)$ the fiber at $(a,b)$ of the map
    \begin{equation*}
         X_1\xrw{(\pt_0,\pt_1)} X_0\times X_0 \xrw{\zg\times \zg}  X^d_0\times X^d_0\;.
    \end{equation*}
    We say that a map $f:X\rw Y$ in $\catwg{n}$ is an $n$-equivalence if
    \begin{itemize}
      \item [i)] For all $a,b\in X_0^d$
      \begin{equation*}
        f(a,b): X(a,b) \rw Y(fa,fb)
      \end{equation*}
      is an $(n-1)$-equivalence.\mk

      \item [ii)] $\p{n}f$ is an $(n-1)$-equivalence.
    \end{itemize}
    This completes the inductive step in the definition of $\catwg{n}$.
\end{definition}
\begin{remark}\label{rem-n-equiv}
    It follows by Definition \ref{def-n-equiv}, Definition \ref{def-hom-dis-ncat} and Proposition \ref{pro-Pa1} that $\cathd{n}\subset \catwg{n}$.
\end{remark}
\subsection{Properties of weakly globular \nfol categories}\label{sbs-prop-wg-nf-cat}
In this section we discuss the main properties of weakly globular \nfol categories. In Proposition \ref{pro-nequiv-to-obj} we show that a weakly globular \nfol category $n$-equivalent to a homotopically discrete one is homotopically discrete. This generalizes to higher dimension the fact that  a category equivalent an equivalence relation is an equivalence relation. We deduce in Corollary \ref{cor-crit-nequiv-rel} a criterion for a weakly globular \nfol category to be homotopically discrete.

The main result of this section, Proposition \ref{pro-crit-ncat-be-wg} b), gives a criterion for an \nfol category to be weakly globular.  This result will be used crucially in the proof of Proposition \ref{pro-transf-wg-struc} to characterize \nfol categories levelwise equivalent to Segalic pseudo-functors. This leads to the main result Theorem \ref{the-strict-funct} on the strictification of Segalic pseudo-functors.

The proof of Proposition \ref{pro-crit-ncat-be-wg} b) uses an inductive argument in conjunction with the proof of a property of the category $\catwg{n}$ (Proposition \ref{pro-crit-ncat-be-wg} a)): the fact that the nerve functor in direction 2, when applied to $\catwg{n}$ takes values in $\funcat{}{\catwg{n-1}}$.

\begin{definition}\label{def-pn}
    For each $1\leq j \leq n$ denote
    \begin{align*}
        \p{j,n} & = \p{j}\p{j-1}\cdots \p{n}:\catwg{n}\rw \catwg{j-1}\\
        \p{n,n}& = \p{n}\;.
    \end{align*}
\end{definition}
\begin{lemma}\label{lem-prop-pn}
    For each $X\in\catwg{n}$, $1\leq j < n$ and $s\geq 2$ it is
    \begin{equation}\label{eq-lem-prop-pn}
    \begin{split}
        &\p{j,n-1}X_s \cong  \p{j,n-1}(\pro{X_1}{X_0}{s})=\\
         & =\pro{\p{j,n-1} X_1}{\p{j,n-1} X_0}{s}\;.
    \end{split}
    \end{equation}
\end{lemma}
\begin{proof}
Since $X\in\catwg{n}$ by definition $\p{n}X\in\catwg{n-1}$, hence
\begin{equation*}
    \p{n-1}(\pro{X_1}{X_0}{s})= \pro{\p{n-1} X_1}{\p{n-1} X_0}{s}
\end{equation*}
which is \eqref{eq-lem-prop-pn} for $j=n-1$. Since $\p{j+1,n}X \in\catwg{j}$ for $1\leq j\leq (n-1)$,  its Segal maps are isomorphisms. Further for all $s\geq 0$
\begin{equation*}
    (\p{j+1}...\p{n}X)_s = \p{j}...\p{n-1}X_s = \p{j,n-1} X_s
\end{equation*}
with $X_s=\pro{X_1}{X_0}{s}$ for $s\geq 2$. This proves \eqref{eq-lem-prop-pn}.
\end{proof}
\begin{remark}\label{rem-eq-def-wg-ncat}
    It follows immediately from Lemma \ref{lem-prop-pn} that if $X\in\catwg{n}$, for all $s\geq 2$
    \begin{equation}\label{eq1-rem-eq-def-wg-ncat}
        X_{s0}^d=(\pro{X_{10}}{X_{00}}{s})^d=\pro{X^d_{10}}{X^d_{00}}{s}\;.
    \end{equation}
    In fact, by \eqref{eq1-rem-eq-def-wg-ncat} in the case $j=2$, taking the 0-component, we obtain
    \begin{equation*}
    \begin{split}
        & \p{1,n-2}(\pro{X_{10}}{X_{00}}{s})= \\
        =\ & \pro{\p{1,n-2}X_{10}}{p...\p{n-2}X_{00}}{s}
    \end{split}
    \end{equation*}
    which is the same as \eqref{eq1-rem-eq-def-wg-ncat}.
\end{remark}
The following proposition is a higher dimensional generalization of the fact that, if a category is equivalent to an equivalence relation, it is itself an equivalence relation.
\begin{proposition}\label{pro-nequiv-to-obj} %5.6
    Let $f:X\rw Y$ be a $\nequ$ in $\catwg{n}$ with $Y\in\cathd{n}$, then $X\in\cathd{n}$.
\end{proposition}
\begin{proof}
By induction on $n$. It is clear for $n=1$. Suppose it is true for $n-1$ and let $f$ be as in the hypothesis. Then $\p{n}f:\p{n} X\rw\p{n}Y$ is a $\equ{n-1}$ with $\p{n}Y\in\cathd{n-1}$ since $Y\in\cathd{n}$. It follows by induction hypothesis that $\p{n}X\in\cathd{n-1}$. We have
\begin{equation}\label{eq1-pro-nequiv}
    X_1=\uset{a,b\in X_0^d}{\cop}X(a,b)\;.
\end{equation}
Since $f$ is a $\nequ$, there are $\equ{n-1}$s
\begin{equation*}
    f(a,b):X(a,b)\rw Y(fa,fb)
\end{equation*}
where $Y(fa,fb)\in\cathd{n-1}$ since $Y\in\cathd{n}$. By induction hypothesis, it follows that $X(a,b)\in\cathd{n-1}$. From \eqref{eq1-pro-nequiv} and the fact that $\cathd{n-1}$ is closed under coproducts (see \cite{Pa1}, Lemma 3.8 a)), we conclude that $X_1\in\cathd{n-1}$.

Since $X\in \catwg{n}$, the induced Segal maps
\begin{equation*}
    \hmu{s}:X_s=\pro{X_1}{X_0}{s}\rw \pro{X_1}{X^d_0}{s}
\end{equation*}
is a $\equ{n-1}$. Since, from above, $X_1$ is homotopically discrete and $X_0^d$ is discrete, by Lemma \cite[Lemma 3.8 c]{Pa1},
\begin{equation*}
  \pro{X_1}{X^d_0}{s}\in\cathd{n-1}
\end{equation*}

Thus by induction hypothesis applied to the induced Segal map $\hmu{s}$ we conclude that $X_s\in\cathd{n-1}$ for all $s\geq 0$.

In summary, we showed that $X\in\catwg{n}$ is such that $X_s\in\cathd{n-1}$ for all $s\geq 0$ and $\p{n}X\in\cathd{n-1}$. Therefore, by definition, $X\in\cathd{n}$.
\end{proof}
\begin{corollary}\label{cor-crit-nequiv-rel}
    Let $X\in\catwg{n}$ be such that $X_1$ and $\pn X$ are in $\cathd{n-1}$. Then $X\in\cathd{n}$.
\end{corollary}
\begin{proof}
Since $X\in\catwg{n}$, the induced Segal maps
\begin{equation*}
    \hmu{s}: X_s\rw \pro{X_1}{X^d_0}{s}
\end{equation*}
are $\equ{n-1}$s for all $s\geq 2$. Since by hypothesis $X_1\in\cathd{n-1}$ and $X_0^d$ is discrete, by (see \cite{Pa1}, Lemma 3.8 a)),
 \begin{equation*}
    \pro{X_1}{X^d_0}{s}\in\cathd{n-1}\;.
\end{equation*}
By Proposition \ref{pro-nequiv-to-obj} applied to $\hmu{s}$ we conclude that $X_s\in\cathd{n-1}$  for all $s\geq 2$. Therefore $X\in\catwg{n}$ is such that $X_s\in \cathd{n-1}$ and $p^{(n)}X\in\cathd{n-1}$. By definition then $X\in\cathd{n}$.
\end{proof}
\begin{corollary}\label{cor2-crit-nequiv-rel}
    Let $X\in\catwg{n}$, then $X\in\cathd{n}$ if and only if there is an $n$-equivalence $\zg:X\rw Y$ with $Y$ discrete.
\end{corollary}
\begin{proof}
If $X\in\cathd{n}$ then by Lemma \cite[Lemma 3.8 a]{Pa1}, $\zg\lo{n}:X\rw X^d$ is an $n$-equivalence. Conversely, suppose that there is an $n$-equivalence $\zg:X\rw Y$ with $Y$ discrete, then in particular $Y\in\cathd{n}$ so, by Proposition \ref{pro-nequiv-to-obj}, $X\in\cathd{n}$.
\end{proof}

\begin{definition}\label{def-kdir-wg-ncat}
Given $X\in\cat{n}$ and $k\geq 0$, let $\Nu{2}X\in\funcat{}{\cat{n-1}}$ as in Definition \ref{def-ner-func-dirk} so that for each $k\geq 0$, $(\Nu{2}X)_k=X_k\up{2}\in \funcat{}{\cat{n-2}}$ is given by
    \begin{equation*}
        (X_k\up{2})_s =
        \left\{
          \begin{array}{ll}
            X_{0k}, & \hbox{$s=0$;} \\
            X_{1k}, & \hbox{$s=1$;} \\
            X_{sk}=\pro{X_{1k}}{X_{0k}}{s}, & \hbox{$s\geq 2$.}
          \end{array}
        \right.
    \end{equation*}
\end{definition}
We denote by $\Nb{n}\catwg{n}$ the image of the multinerve functor $\Nb{n}:\catwg{n}\rw \funcat{n-1}{\Cat}$. Note that, since $\Nb{n}$ is fully faithful, we have an isomorphism $\catwg{n}\cong \Nb{n}\catwg{n}$.

The following lemmas are needed in the initial steps in the induction in the proof of Proposition \ref{pro-crit-ncat-be-wg}.

\begin{lemma}\label{lem-crit-doucat-wg}
    Let $X\in\cat{2}$ be such that
    \begin{itemize}
      \item [i)] $X_0\in\cathd{}$,\mk

      \item [ii)] $\bar p J_2 X\in N\Cat$.\mk
    \end{itemize}
    Then $X\in \catwg{2}$.
\end{lemma}
\begin{proof}
Since $X_0\in \cathd{}$, $pX_0 = X_0^d$. By hypothesis, $p X_2\cong \tens{p X_1}{p X_0}$ and $X_2\cong \tens{X_1}{X_0}$. Using the fact that $p$ commutes with pullbacks over discrete objects, we obtain
\begin{equation*}
    p(\tens{X_1}{X_0})\cong pX_2\cong \tens{p X_1}{p X_0}=\tens{p X_1}{p X^d_0}=p(\tens{X_1}{X^d_0})\;.
\end{equation*}
This shows that the map
\begin{equation*}
    \hmu2:\tens{X_1}{X_0} \rw \tens{X_1}{X^d_0}
\end{equation*}
is essentially surjective on objects. On the other hand, this map is also fully faithful. In fact, given $(a,b),(c,d)\in \tens{X_{10}}{X_{00}}$, we have
\begin{equation*}
\begin{split}
    & (\tens{X_1}{X_0})((a,b),(c,d))\cong X_1(a,c)\tiund{X_0(\pt_0 a,\pt_0 c)} X_1(b,d)\cong \\
    & \cong X_1(a,c)\times X_1(b,d) \cong (\tens{X_1}{X^d_0})(\hmu{2}(a,b),\hmu{2}(c,d))
\end{split}
\end{equation*}
where we used the fact that $X_0(\pt_0 a,\pt_0 c)$ is the one-element set, since $X_0\in\cathd{}$. We conclude that $\hmu{2}$ is an equivalence of categories.

Similarly one shows that for all $k\geq 2$
\begin{equation*}
    \hmuk : \pro{X_1}{X_0}{k} \rw \pro{X_1}{X^d_0}{k}
\end{equation*}
is an equivalence of categories. By definition, this means that $X\in\catwg{2}$.
\end{proof}
\begin{lemma}\label{lem-crit-3cat-wg}\
\begin{itemize}
  \item [a)] The functor $\Nu{2}:\cat{3}\rw \funcat{}{\cat{2}}$ restricts to $\Nu{2}:\catwg{3}\rw \funcat{}{\catwg{2}}$.\mk

  \item [b)] Let $X\in\cat{3}$ be such that
  \begin{itemize}
    \item [i)] $X_0\in\cathd{2}$, $X_{s0}\in\cathd{}$.

    \item [ii)] $\bar p J_3 X\in \Nb{2}\catwg{2}$.
  \end{itemize}
  \mk

 \nid Then $X\in\catwg{3}$.
\end{itemize}

\end{lemma}
\begin{proof}\

\nid a) Since $X\in\catwg{3}$, $X_0\in\cathd{2}$ therefore for each $k\geq 0$
\begin{equation*}
    (\Nu{2}X)_{k0}=X_{0k}\in\cathd{}\;.
\end{equation*}
Further, since $\p{3}X\in\catwg{2}$, for each $k\geq 0$
\begin{equation*}
    \ovl{p}J_2(\Nu{2}X)_k=(\p{3}X)_k\up{2}
\end{equation*}
is the nerve of a category. It follows from Lemma \ref{lem-crit-doucat-wg} applied to $(\Nu{2}X)_k$ that $(\Nu{2}X)_k\in\catwg{2}$.

\bk

\nid b) By hypothesis, $X_s\in\cat{2}$ is such that $X_{s0}\in\cathd{}$ and $\bar p J_2 X_s$ is the nerve of a category. Thus by  Lemma \ref{lem-crit-doucat-wg}, $X_s\in\catwg{2}$.

Also by hypothesis $\p{3}X\in\catwg{2}$. To show that $X\in\catwg{3}$ it remains to prove that the map
\begin{equation*}
    \hmu{s}:\pro{X_1}{X_0}{s}\rw \pro{X_1}{X^d_0}{s}
\end{equation*}
is a 2-equivalence. We first show this for $s=2$, the case $s>2$ being similar. We first show that it is a local equivalence. By part a) $(\Nu{2}X)_1\in \catwg{2}$. Thus there is an equivalence of categories
\begin{equation}\label{eq1-lem-crit-3cat}
    \tens{X_{11}}{X_{01}}\rw \tens{X_{11}}{X^d_{01}}=\tens{X_{11}}{(\di{2}\p{2}X_0)_1}\;.
\end{equation}
From hypothesis ii) by Remark \ref{rem-eq-def-wg-ncat} using the fact that $p\p{2}X_{s0}=X_{so}^d$ we have
\begin{equation*}
   X_{20}^{d}=(\tens{X_{10}}{X_{00}})^d \cong \tens{X^d_{10}}{X^d_{00}}\;.
\end{equation*}
Let $(a,b),(c,d)\in \tens{X^d_{10}}{X^d_{00}}$\;. By \eqref{eq1-lem-crit-3cat} there is an equivalence of categories
\begin{equation}\label{eq2-lem-crit-3cat}
\begin{split}
    & (\tens{X_1}{X_0})((a,b),(c,d)) = \\
    & = X_1(a,c) \tiund{X_0(\pt_0 a, \pt_0 c)} X_1 (b,d) \rw X_1(a,c) \tiund{(\di{2}\p{2}X_0)(\tilde\pt_0 a, \tilde\pt_0 c)} X_1 (b,d)
\end{split}
\end{equation}
On the other hand, since $\p{2}X_0 \in \cathd{}$, $\p{2}X_0(\tilde\pt_0 a, \tilde\pt_0 c)$ is the one-element set. Therefore
\begin{equation}\label{eq2A-lem-crit-3cat}
\begin{split}
    &  X_1(a,c) \tiund{(\di{2}\p{2}X_0)(\tilde\pt_0 a, \tilde\pt_0 c)} X_1 (b,d) = \\
    & = X_1(a,c)\times X_1(b,d)=(\tens{X_1}{X^d_0})((a,b),(c,d))\;.
\end{split}
\end{equation}
From \eqref{eq2-lem-crit-3cat} and \eqref{eq2A-lem-crit-3cat} we conclude that $\hmu{2}$ is a local equivalence. Further, by hypothesis ii), there is an equivalence of categories
\begin{equation*}
\begin{split}
    & \p{2}\hmu{2}:\p{2}(\tens{X_1}{X_0})= \tens{\p{2} X_1}{\p{2} X_0} \xrw{\sim} \\
    & \rw \tens{\p{2} X_1}{(\p{2} X_0)^d}=\p{2}(\tens{X_1}{X^d_0})\;.
\end{split}
\end{equation*}
In conclusion, $\hmu{2}$ is a 2-equivalence, as required.
\end{proof}
\begin{proposition}\label{pro-crit-ncat-be-wg}\

    \begin{itemize}
      \item [a)] The functor $\Nu{2}:\cat{n}\rw\funcat{}{\cat{n-1}}$ restricts to a functor $\Nu{2}:\catwg{n}\rw \funcat{}{\catwg{n-1}}$.\mk

      \item [b)] Let $X\in\cat{n}$ be such that\mk
  \begin{itemize}
    \item [i)] $X_0\in\cathd{n-1}$, $X_{s0}\in\cathd{n-2}$.\mk

    \item [ii)] $\bar p J_n X\in \Nb{n-1}\catwg{n-1}$ for all $s \geq 0$.\mk
  \end{itemize}
  \mk
    \end{itemize}
    Then $X\in\catwg{n}$.
\end{proposition}
\begin{proof}
By induction on $n$. For $n=2,3$ see Lemmas \ref{lem-crit-doucat-wg} and \ref{lem-crit-3cat-wg}. Suppose, inductively, that it holds for $(n-1)$.
\bk

a) Clearly $X_k\up{2}\in\cat{n-1}$; we show that $X_k\up{2}$ satisfies the inductive hypothesis b) and thus conclude that $X_k\up{2}\in\catwg{n-1}$.

We have $(X_k\up{2})_0=X_{0k}\in\cathd{n-2}$ since $X_0\in\cathd{n-1}$ (as $X\in\catwg{n}$). Further,
\begin{equation*}
    (X_k\up{2})_{s0}=X_{sk0}\in \cathd{n-3}
\end{equation*}
since $X_{sk}\in\catwg{n-2}$ (as $X_s\in \catwg{n-1}$ because $X\in\catwg{n}$). Thus condition i) in the inductive hypothesis b)  holds for $X_{k}\up{2}$. To show that condition ii) holds, notice that
\begin{equation}\label{eq1-lem-xk-catwg}
    \bar p J_{n-1}X_k\up{2}=\Nb{n-1}(\p{n}X)_k\up{2}
\end{equation}
In fact, for all $(r_1,...,r_{n-2})\in {{\zD}^{n-2}}^{op}$,
\begin{align*}
    &(\bar p J_{n-1}X_k\up{2})_{r_1...r_{n-2}}=p(X_k\up{2})_{r_1...r_{n-2}}=\\
    =\ & pX_{r_1,k,r_2...r_{n-2}}=(\bar p J_{n-2}X_{r,k})_{r_2...r_{n-2}}=\\
    =\ & (\Nb{n-2}\p{n-2}X_{r,k})_{r_2...r_{n-2}}=\\
    =\ & (\Nb{n-2}((\p{n}X)\up{2}_k)_{r_1})_{r_2...r_{n-2}}=\\
    =\ & (\Nb{n-1}(\p{n}X)\up{2}_k)_{r_1...r_{n-2}}\;.
\end{align*}
Since this holds for all $r_1,...,r_{n-2}$, \eqref{eq1-lem-xk-catwg} follows

By induction hypothesis a) applied to $\p{n}X$, $(\p{n}X)_k\up{2}\in\catwg{n-2}$. Therefore \eqref{eq1-lem-xk-catwg} means that $X_k\up{2}\in\cat{n-1}$ satisfies condition ii) in the inductive hypothesis b). Thus we conclude that $X_k\up{2}\in\catwg{n-1}$ proving a).
\bk

b) Suppose, inductively, that the statement holds for $n-1$ and let $X$ be as in the hypothesis. For each $s\geq 0$ consider $X_s\in\cat{n-1}$. By hypothesis, $X_{s0}\in\cathd{n-2}$ and
\begin{equation*}
    \bar p J_{n-1}X_s =(\bar p J_n X)_s \in \Nb{n-2}\catwg{n-2}
\end{equation*}
since $\bar p J_n X\in \Nb{n-1}\catwg{n-1}$. Thus $X_s$ satisfies the induction hypothesis and we conclude that $X_s\in\catwg{n-1}$. Further, for each $k_1,\ldots,k_{n-2}$ we have
\begin{equation*}
\begin{split}
    & (\Nb{n-1}\ovl{\p{n-1}}X)_{k_1...k_{n-2}} = (\Nb{n-2}\p{n-1}X_{k_1})_{k_2...k_{n-2}}=\\
    =\;& (\bar p J_{n-1}X_{k_1})_{k_2...k_{n-2}}=p X_{k_1...k_{n-2}}=(\bar p J_n X)_{k_1...k_{n-2}}\;.
\end{split}
\end{equation*}
Since, by hypothesis, $\bar p J_n X\in \Nb{n-1}\catwg{n-1}$, we conclude that $\ovl{\p{n-1}}X\in\catwg{n-1}$. We can therefore define $\p{n}X = \ovl{\p{n-1}}X\in\catwg{n-1}$.

To prove that $X\in\catwg{n}$ it remains to prove that the induced Segal maps
\begin{equation*}
    \hmu{s}:\pro{X_1}{X_0}{s}\rw \pro{X_1}{X_0^d}{s}
\end{equation*}
are $\equ{n-1}$s for all $s\geq 2$. We prove this for $s=2$ the case $s>2$ being similar. We claim that $X_k\up{2}\in\cat{n-1}$ satisfies the inductive hypothesis b). In fact, $(X_k\up{2})_0 = X_{0k}\in \cathd{n-2}$ since $X_0\in\cathd{n-1}$; for each $s\geq 0$, $(X_k\up{2})_{s0}=X_{sk0}\in\cathd{n-3}$ since, from above, $X_s\in\catwg{n-1}$.

Also, from a) and the fact that, by hypothesis, $\bar p J_n X\in\Nb{n-1}\catwg{n-1}$, we conclude that
\begin{equation*}
    \bar p J_{n-1} X_k\up{2} \in \Nb{n-2}\catwg{n-2} \;.
\end{equation*}
Thus $X_k\up{2}$ satisfies the inductive hypothesis b) and we conclude that $X_k\up{2}\in\catwg{n-1}$. It follows that the induced Segal map
\begin{equation}\label{eq3-crit-ncat-be-wg}
    \tens{X_{1k}}{X_{0k}}\rw \tens{X_{1k}}{X^d_{0k}}=\tens{X_{1k}}{(\p{2,n-1}X_0)_k}
\end{equation}
is a $\equ{n-2}$. Since $\p{n}X\in\catwg{n-1}$, using Remark \ref{rem-eq-def-wg-ncat} and the fact that $(\p{n}X)_{s0}^d = (\p{n-2}\p{n-1}X_{s0})^d = X_{s0}^d$ we obtain
\begin{equation*}
    (\tens{X_{10}}{X_{00}})^d=\tens{X^d_{10}}{X^d_{00}}\;.
\end{equation*}
Let $(a,b),(c,d)\in(\tens{X_{10}}{X_{00}})^d=\tens{X^d_{10}}{X^d_{00}}$. By \eqref{eq3-crit-ncat-be-wg} there is a $(n-2)$-equivalence
\begin{equation}\label{eq3A-crit-ncat-be-wg}
\begin{split}
    & (\tens{X_1}{X_0})((a,b),(c,d))=X_1(a,c)\tiund{X_0(\pt_0 a,\pt_0 c)}X_1(b,d)\rw \\
    & \rw X_1(a,c)\tiund{(\p{2,n-1}X_0)(\tilde\pt_0 a,\tilde\pt_0 c)}X_1(b,d)\;.
\end{split}
\end{equation}
On the other hand, $\p{2,n-1}X_0\in\cathd{}$ is an equivalence relation, therefore
\begin{equation*}
    \p{2,n-1}X_0(\tilde\pt_0 a,\tilde\pt_0 c)
\end{equation*}
is the one-element set. It follows that
\bk
\begin{equation}\label{eq3B-crit-ncat-be-wg}
\begin{split}
    & X_1(a,c)\tiund{(\p{2,n-1}X_0)(\tilde\pt_0 a,\tilde\pt_0 c)}X_1(b,d)\cong\\
    & \cong X_1(a,c)\times X_1(b,d)\cong (\tens{X_1}{X_0^d})((a,b),(c,d))\;.
\end{split}
\end{equation}
Thus \eqref{eq3A-crit-ncat-be-wg}  and \eqref{eq3B-crit-ncat-be-wg} imply that $\hmu{2}$ is a local $\equ{n-2}$.

To show that $\hmu{2}$ is a $\equ{n-1}$ it remains to prove that $\p{n-1}\hmu{2}$ is a $\equ{n-2}$. Since from above, $\p{n}X=\ovl{\p{n-1}}X\in\catwg{n-1}$, we have
\begin{equation*}
\begin{split}
    & \p{n-1}\hmu{2}:\p{n-1}(\tens{X_1}{X_0})\cong  \tens{\p{n-1}X_1}{\p{n-1}X_0}\rw\\
    & \rw \tens{\p{n-1}X_1}{(\p{n-1}X_0)^d}=\p{n-1}(\tens{X_1}{X^d_0})\;.
\end{split}
\end{equation*}
is a $\equ{n-2}$, as required.
\end{proof}
\section{Segalic pseudo-functors and their strictification}\label{sec-pseud-strict}
In this section we introduce the category $\segpsc{n-1}{\Cat}$ of Segalic pseudo-functors and we prove in Theorem \ref{the-strict-funct} that the strictification of a Segalic pseudo-functor is a weakly globular \nfol category.

After giving the definition of Segalic pseudo-functors in \ref{def-seg-ps-fun}, we show in Proposition \ref{pro-transf-wg-struc} that if an \nfol category, viewed as a diagram in $\funcat{n-1}{\Cat}$ is levelwise equivalent to a Segalic pseudo-functor, then it is weakly globular. We show there that the strictification machinery, when applied to a Segalic pseudo-functor, produces an \nfol category  satisfying  the hypotheses of Proposition \ref{pro-transf-wg-struc} (when viewed as a diagram in $\funcat{n-1}{\Cat}$). We conclude that its strictification is a weakly globular \nfol category.\bk

\subsection{The idea of Segalic pseudo-functors}\label{sbs-seg-map-pse-fun}
The category of Segalic pseudo-functors is a full subcategory of the category $\psc{n-1}{\Cat}$ of pseudo-functors and pseudo-natural transformations \cite{Borc}.

A topological intuition about an object of $\psc{n-1}{\Cat}$ is that it consists of categories $X_{\uk}$ for each object $\uk$ of $\dop{n-1}$ together with multi-simplicial face and degeneracy maps satisfying the multi-simplicial identities not as equalities but as isomorphisms, and these isomorphisms satisfy coherence axioms. Guided by this intuition, we generalize to certain pseudo-functors the multi-simpicial notion of Segal map.

For this purpose, consider a  functor $H\in\funcat{n-1}{\Cat}$. For each $\uk=(k_1,...,k_{n-1})\in\dop{n-1}$ and $1\leq i\leq n-1$ we have
\begin{equation*}
    H(k_1,...,k_{i-1},\mi,k_{i+1},...,k_{n-1})\in\funcat{}{\Cat}
\end{equation*}
and there is a corresponding Segal map for each $k_i\geq 2$
\begin{equation}\label{eq1-sbs-seg-map}
    H_{\uk}\rw \pro{H_{\uk(1,i)}}{H_{\uk(0,i)}}{k_i}
\end{equation}
identified by the commuting diagram
\begin{equation}\label{eq2-sbs-seg-map}
    \xy
    0;/r.60pc/:
    (0,0)*+{H_{\uk}}="1";
    (-7,-5)*+{H_{\uk(1,i)}}="2";
    (-2,-5)*+{H_{\uk(1,i)}}="3";
    (9,-5)*+{H_{\uk(1,i)}}="4";
    (-10,-9)*+{H_{\uk(0,i)}}="5";
    (-5,-9)*+{H_{\uk(0,i)}}="6";
    (0,-9)*+{H_{\uk(0,i)}}="7";
    (6,-9)*+{H_{\uk(0,i)}}="8";
    (12,-9)*+{H_{\uk(0,i)}}="9";
    (3,-5)*+{\cdots}="10";
    (3,-9)*+{\cdots}="11";
    {\ar_{\nu_1}"1";"2"};
    {\ar^{\nu_2}"1";"3"};
    {\ar^{\nu_k}"1";"4"};
    {\ar_{d_1}"2";"5"};
    {\ar^{d_0}"2";"6"};
    {\ar^{d_1}"3";"6"};
    {\ar^{d_0}"3";"7"};
    {\ar_{d_1}"4";"8"};
    {\ar^{d_0}"4";"9"};
    \endxy
\end{equation}
If $H$ is not a  functor but a pseudo-functor $H\in\psc{n-1}{\Cat}$, diagram \eqref{eq2-sbs-seg-map} no longer commutes but pseudo-commutes and thus we can no longer define Segal maps.
However, if $H_{\uk(0,i)}$ is a discrete category, then diagram  \eqref{eq2-sbs-seg-map} commutes and therefore we can define Segal maps for $H$.

In the definition of Segalic pseudo-functor we require the above discreetness conditions to be satisfied so as to be able to speak about Segal maps and then we require all Segal maps to be isomorphisms.

The last condition in the definition of Segalic pseudo-functor is about the existence of a truncation functor. Applying the isomorphism classes of objects functor $p:\Cat\rw\Set$ to any pseudo-functor in $\psc{n}{\Cat}$ produces a strict functor; that is, there is a functor
\begin{equation*}
    \ovl{p}:\psc{n}{\Cat}\rw\funcat{n-1}{\Set}
\end{equation*}
such that, for all $\uk\in\dop{n-1}$
\begin{equation*}
    (\ovl{p}X)_{\uk}=pX_{\uk}\;.
\end{equation*}
For $X$ to be a Segalic pseudo-functor we require $\ovl{p}X$ to be a weakly globular \nfol category (more precisely, in the image of $J_n:\catwg{n}\rw\funcat{n-1}{\Cat}$). Thus we have a functor
\begin{equation*}
    \p{n+1}:\segpsc{n}{\Cat}\rw\catwg{n}\;.
\end{equation*}
\subsection{Segal maps for special pseudo-functors.}\label{sbs-seg-map-spec-ps}
Let $H\in\Ps\funcat{n}{\Cat}$ be such that $H_{\uk(0,i)}$ is discrete for all $\uk\in\Dmenop$ and all $i\geq 0$. Then the following diagram commutes, for each $k_i\geq 2$.
\begin{equation*}
    \xy
    0;/r.60pc/:
    (0,0)*+{H_{\uk}}="1";
    (-7,-5)*+{H_{\uk(1,i)}}="2";
    (-2,-5)*+{H_{\uk(1,i)}}="3";
    (9,-5)*+{H_{\uk(1,i)}}="4";
    (-10,-9)*+{H_{\uk(0,i)}}="5";
    (-5,-9)*+{H_{\uk(0,i)}}="6";
    (0,-9)*+{H_{\uk(0,i)}}="7";
    (6,-9)*+{H_{\uk(0,i)}}="8";
    (12,-9)*+{H_{\uk(0,i)}}="9";
    (3,-5)*+{\cdots}="10";
    (3,-9)*+{\cdots}="11";
    {\ar_{\nu_1}"1";"2"};
    {\ar^{\nu_2}"1";"3"};
    {\ar^{\nu_k}"1";"4"};
    {\ar_{d_1}"2";"5"};
    {\ar^{d_0}"2";"6"};
    {\ar^{d_1}"3";"6"};
    {\ar^{d_0}"3";"7"};
    {\ar_{d_1}"4";"8"};
    {\ar^{d_0}"4";"9"};
    \endxy
\end{equation*}
There is therefore a unique Segal map
\begin{equation*}
    H_{\uk}\rw \pro{H_{\uk(1,i)}}{H_{\uk(0,i)}}{k_i}\;.
\end{equation*}
\subsection{The definition of Segalic pseudo-functor.}\label{sbs-def-seg-ps}
\begin{definition}\label{def-seg-ps-fun}
    We define the subcategory $\segpsc{n}{\Cat}$ of $\psc{n}{\Cat}$ as follows:

    For $n=1$, $H\in \segpsc{}{\Cat}$ if $H_0$ is discrete and the Segal maps are isomorphisms: that is, for all $k\geq 2$
    \begin{equation*}
        H_k\cong\pro{H_1}{H_0}{k}
    \end{equation*}
    Note that, since $p$ commutes with pullbacks over discrete objects, there is a functor
   \begin{equation*}
   \begin{split}
       & \p{2}:\segpsc{}{\Cat} \rw \Cat\;, \\
       & (\p{2}X)_{k}=p X_k\;.
   \end{split}
   \end{equation*}
   That is the following diagram commutes:
   \begin{equation*}
    \xymatrix{
    \segpsc{}{\Cat} \ar@{^{(}->}[rr]\ar_{\p{2}}[d] && \psc{}{\Cat} \ar^{\ovl{p}}[d]\\
    \Cat \ar[rr] && \funcat{}{\Set}
    }
   \end{equation*}
   When $n>1$, $\segpsc{n}{\Cat}$ is the full subcategory of $\psc{n}{\Cat}$ whose objects $H$ satisfy the following: \mk
   \begin{itemize}
     \item [a)] \emph{Discreteness condition}:  $H_{\uk(0,i)}$ is discrete for all $\uk\in\Dmenop$ and $1 \leq i \leq n$.\mk
     \item [b)] \emph{Segal condition}: All Segal maps are isomorphisms
      \begin{equation*}
      H_{\uk}\cong\pro{H_{\uk(1,i)}}{H_{\uk(0,i)}}{k_i}
      \end{equation*}
      for all $\uk\in\Dmenop$, $1 \leq i \leq n$ and $k_i\geq2$.\mk
     \item [c)] \emph{Truncation functor}: There is a functor
     \begin{equation*}
        \p{n+1}:\segpsc{n}{\Cat}\rw \catwg{n}
     \end{equation*}
   \end{itemize}
    making the following diagram commute:
    \begin{equation*}
        \xymatrix{
        \segpsc{n}{\Cat} \ar@{^(->}[r] \ar_{\p{n+1}}[d] & \psc{n}{\Cat} \ar^{\ovl{p}}[d]\\
        \catwg{n} \ar_{\Nb{n}}[r] & \funcat{n}{\Set}
        }
    \end{equation*}
\end{definition}
\begin{lemma}\label{lem-seg-pse-fun}
    Let $X\in\segpsc{n}{\Cat}$\; $n\geq 2$. Then for each $j\geq 0$ $X_{j*}\in\segpsc{n-1}{\Cat}$.
\end{lemma}
\begin{proof}
By induction on $n$. Let $X\in\segpsc{2}{\Cat}$. Since $X\in\psc{2}{\Cat}$, for each $j\geq 0$ $X_{j*}\in\psc{}{\Cat}$. By definition of Segalic pseudo-functor, $X_{j0}$ is discrete and for each $r\geq 2$
\begin{equation*}
    X_{jr}\cong \pro{X_{j1}}{X_{j0}}{r}\;.
\end{equation*}
By definition this means that $X_{j*}\in\segpsc{}{\Cat}$. Suppose, inductively, that the lemma holds for $(n-1)$ and let $X\in\segpsc{n}{\Cat}$. For each $j\geq 0$, $X_{j*}\in\psc{n-1}{\Cat}$.

Given $\ur\in\dop{n-1}$ denote $\uk=(j,\ur)\in\dop{n}$. Then for each $1 \leq i\leq n-2$,
\begin{equation*}
    X_{\uk(i+1,0)}=(X_j)_{\ur(i,0)}
\end{equation*}
is discrete since $X\in\segpsc{n}{\Cat}$; further, by hypothesis there are isomorphisms:
\begin{align*}
    (X_j)_{\ur} = X_{\uk}\cong & \pro{X_{\uk(i+1,1)}}{X_{\uk(i+1,0)}}{k_{i+1}}\cong\\
    \cong & \pro{(X_j)_{\ur(i,1)}}{(X_j)_{\ur(i,0)}}{r_i}\;.
\end{align*}
To show that $X_j\in\segpsc{n-1}{\Cat}$ it remains to show that $\p{n}X_j\in\catwg{n-1}$ where
\begin{equation*}
    (\p{n}X_j)_{\ur}=p X_{j\ur}
\end{equation*}
for each $\ur\in\dop{n-1}$.

Since $X\in\segpsc{n}{\Cat}$, by definition $\p{n+1}X\in\catwg{n}$ where $(\p{n+1}X)_{\uk}=p X_{\uk}$ for all $\uk\in\dop{n}$. We also observe that, for each $j\geq 0$
\begin{equation}\label{eq-lem-seg-pse-fun}
    (\p{n+1}X)_j=\p{n}X_j
\end{equation}
since, for each $\ur\in\dop{n-1}$,
\begin{equation*}
    (\p{n+1}X)_{j\ur}=p X_{j\ur}=(\p{n}X_j)_{\ur}\;.
\end{equation*}
Since $\p{n+1}X\in\catwg{n}$, $(\p{n+1}X)_j\in\catwg{n-1}$ so by \eqref{eq-lem-seg-pse-fun} we conclude that $\p{n}X_j\in\catwg{n-1}$ as required.
\end{proof}
\subsection{Segalic pseudo-functors and $\mathbf{n}$-fold categories}\label{sbs-seg-pse-nfold-cat}
In the following proposition we show that if an \nfol category is levelwise equivalent (as a diagram in $\funcat{n-1}{\Cat}$) to a Segalic pseudo-functor, then it is a weakly globular \nfol category.
\begin{proposition}\label{pro-transf-wg-struc}
    Let $H\in \segpsc{n-1}{\Cat}$ and let $L\in\Catn$ be such that there is an equivalence of categories $(J_n L)_{\uk}\simeq H_{\uk}$ for all $\uk\in\Dmenop$, then
    \begin{itemize}
      \item [a)] $L\in\catwg{n}$.
      \item [b)] If, further, $H_{\uk}\in \cathd{}$ for all $\uk$ and $\pn H\in \cathd{n-1}$, then $L\in\cathd{n}$.
    \end{itemize}
\end{proposition}
\begin{proof}
By induction on $n$. For $n=2$, if $H\in\segpsc{}{\Cat}$, then by definition $H_0$ is discrete; thus $L_0\in\cathd{}$. Also, since $L_k\simeq H_k$ for all $k\in\Dop$, $p L_k\cong p H_k$, and therefore $\bar p L\cong \bar p H =\p{2} H$ is the nerve of a category. So by Proposition \ref{pro-crit-ncat-be-wg} b) $L\in\catwg{2}$.

If, further, $H_k\in\cathd{}$ for all $k$ and $\p{2}H\in\cathd{}$, then $L_1 \in \cathd{}$ (since $L_1 \sim H_1$) and $\p{2}L =\p{2}H \in \cathd{}$. Therefore, by Corollary \ref{cor-crit-nequiv-rel}, $L\in\cathd{2}$.

Suppose, inductively, that the lemma holds for $n-1$ and let $L$ and $H$ be as in the hypothesis a).

We are going to show that $L\in\cat{n}$ satisfies the hypotheses of Proposition \ref{pro-crit-ncat-be-wg} b) which then implies that $L\in\catwg{n}$.

Let $\ur\in\dop{n-2}$ and denote $\uk=(i,\ur)\in\dop{n-1}$. By hypothesis, there are equivalences of categories
\begin{equation*}
    (J_{n-1}L_i)_{\ur}=J_n L_{\uk}\cong H_{\uk}=(H_i)_{\ur}\;.
\end{equation*}
Since $L\in\cat{n}$, $L_i\in\cat{n-1}$ and since $H\in\segpsc{n}{\Cat}$, by Lemma \ref{lem-seg-pse-fun}, $H_i\in\segpsc{n-1}{\Cat}$.

%Then, by definition, for all $i\geq 0$
%%
%\begin{equation*}
%    H_{i\bl}\in \segpsc{n-2}{\Cat}\;.
%\end{equation*}
%%
%Also, for all $\uk\in \dop{n-2}$, we have equivalences of categories
%%
%\begin{equation*}
%    (J_{n-1}L_{i\bl})_{\uk}=L_{\uk(i,0)}\simeq H_{\uk(i,0)}=(H_{i\bl})_{\uk}\;.
%\end{equation*}
%
Thus $H_{i\bl}$ and $L_{i\bl}$ satisfy the inductive hypothesis a) and we conclude that $L_{i\bl}\in \catwg{n-1}$. In particular, this implies that $L_{i0}\in \cathd{n-2}$. Thus, by Proposition \ref{pro-crit-ncat-be-wg} b), to show that $L\in \catwg{n}$ it is enough to prove that $L_{0\bl}\in\cathd{n-1}$ and that $\bar p J_n L\in \Nb{n-1}\catwg{n-1}$.
We have $(H_{0\bl})_{\uk}=H_{\uk(0,0)}$ discrete and
\begin{equation*}
    \p{n-1}H_{0\bl}=(\p{n}H)_0 \in\cathd{n-2}
\end{equation*}
since by hypothesis $\p{n}H\in\catwg{n}$ as $H\in\segpsc{n-1}{\Cat}$. Thus $H_{0\bl}$ and $L_{0\bl}$ satisfy the inductive hypothesis b) and we conclude that $L_{0\bl}\in\cathd{n-1}$. For each $\uk \in\dop{n-2}$,
\begin{equation*}
    (\bar p J_n L)_{\uk} = p L_{\uk} \cong p H_{\uk} =(\bar p J_n H)_{\uk}\;.
\end{equation*}
Since $\p{n}H\in\catwg{n}$, this means that $\bar p J_n L\in \Nb{n-1}\catwg{n-1}$. By Proposition \ref{pro-crit-ncat-be-wg} b), we conclude that $L\in\catwg{n}$, proving a) at step $n$.

Suppose that $H$ is as in b). By Corollary \ref{cor-crit-nequiv-rel}, to show that $L\in\cathd{n}$, it is enough to show that $L_1 \in\cathd{n-1}$ and $\p{n}L \in \cathd{n-1}$.
For all $\uk\in\dop{n-2}$ there is an equivalence of categories
\begin{equation*}
    (L_{1\bl})_{\uk}\simeq (H_{1\bl})_{\uk}
\end{equation*}
therefore, since by hypothesis, $(H_{1\bl})_{\uk}\in\cathd{}$, also $(L_{1\bl})_{\uk}\in\cathd{}$. Further, since $\p{n}H\in \cathd{n-1}$, then
\begin{equation*}
    \p{n-1}H_{1\bl}=(\p{n}H)_{1\bl}\in \cathd{n-2}\;.
\end{equation*}
Thus $L_{1\bl}$ and $H_{1\bl}$ satisfy induction hypothesis and we conclude that $L_{1\bl}\in\cathd{n-1}$. Finally,
\begin{equation*}
    \p{n}L=\p{n}H\in\cathd{n-1}\;.
\end{equation*}
Thus by Corollary \ref{cor-crit-nequiv-rel} we conclude that $L\in\cathd{n}$.
\end{proof}
\subsection{Strictification of Segalic pseudo-functors}\label{sbs-stric-pse-fun}
In this section we prove our main result, Theorem \ref{the-strict-funct}, that the strictification functor applied to the category of Segalic pseudo-functors gives a weakly globular \nfol category. The strategy to prove this result is to show that the strictification of a Segalic pseudo-functor is an \nfol category and that it satisfies the hypotheses of Proposition \ref{pro-transf-wg-struc}.

\begin{lemma}\label{lem-strict-mon-wg-seg-ps-fun}
    Let $T$ be the monad corresponding to the adjunction given by the forgetful functor
    \begin{equation*}
        U:\funcat{n}{\Cat}\rw [ob(\Dnop),\Cat ]
    \end{equation*}
    and its left adjoint. Let $H\in \segpsc{n}{\Cat}$, then
    \begin{itemize}
      \item [a)] The pseudo $T$-algebra corresponding to $H$ has structure map $h: TUH\rw H$ as follows:
          \begin{equation*}
            (TUH)_{\uk}=\underset{\ur\in\zD^n}{\coprod}\zD^n(\uk,\ur)\times H_{\ur}=\underset{\ur\in\zD^n}{\coprod}\underset{\zD^n(\uk,\ur)}{\coprod} H_{\ur}\;.
          \end{equation*}
          If $f\in\zD^n(\uk,\ur)$, let
          \begin{equation*}
            i_{\ur}=\underset{\zD^n(\uk,\ur)}{\coprod}H_{\ur}\rw \underset{\ur\in\zD^n}{\coprod}\underset{\zD^n(\uk,\ur)}{\coprod} H_{\ur}=(TUH)_{\ur}
          \end{equation*}
          \begin{equation*}
            j_f :H_{\ur}\rw \underset{\zD^n(\uk,\ur)}{\coprod}H_{\ur}\;,
          \end{equation*}
          then
          \begin{equation*}
            h_{\uk}\,i_{\ur}\,j_f= H(f)\;.
          \end{equation*}
      \item [b)] There are functors for each $\uk \in \zD^n$, $0\leq i \leq n$,
      \begin{equation*}
        \pt_{i1}, \pt_{i0}:(TUH)_{\uk(1,i)}\rightrightarrows (TUH)_{\uk(0,i)}
      \end{equation*}
      such that the following diagram commutes
      \begin{equation}\label{fig.pro.strict-mon}
        \xymatrix@R=35pt @C=50pt{
        (TUH)_{\uk(1,i)} \ar^{h_{\uk(1,i)}}[r] \ar^{\pt_{i1}}[d]<1ex> \ar_{\pt_{i0}}[d]<-1ex> & H_{\uk(1,i)} \ar^{d_{i1}}[d]<1ex> \ar_{d_{i0}}[d]<-1ex>  \\
        (TUH)_{\uk(0,i)} \ar_{h_{\uk(0,i)}}[r] & H_{\uk(0,i)}
        }
      \end{equation}
      \item [c)] For all $\uk=(k_1,\ldots, k_n)\in\zD^n$ there are  isomorphisms
      \begin{equation*}
        (TUH)_{\uk}=\pro{(TUH)_{\uk(1,i)}}{(TUH)_{\uk(0,i)}}{k_i}\;.
      \end{equation*}
      \item [d)] The morphism $h_{\uk}:(TUH)_{\uk}\rw H_{\uk}$ is given by
      \begin{equation*}
        h_{\uk}=(h_{\uk(1,i)},\ldots,h_{\uk(1,i)})
      \end{equation*}
    \end{itemize}
\end{lemma}
\begin{proof}\

a) From the general correspondence between pseudo $T$-algebras and pseudo-functors, the pseudo $T$-algebra corresponding to $H$ has a structure map $h:TUH\rw H$ as stated. The rest follows from the fact that, if $X$ is a set and $\clC$ is a category, $X\times \clC \cong \uset{X}{\coprod}\clC$.\bk

b) Let $\nu_j:[0]\rw [1]$, $\nu_j(0)=0$. $\nu_j(1)=i$ for $j=0,1$ and let $\zd_{ij}:\uk(0,i)\rw\uk(1,i)$ be given by
\begin{equation*}
    \zd_{ij}(k_s)=\left\{
    \begin{array}{ll}
     k_s, & \hbox{$s\neq i$;} \\
     \nu_{j}(k_i), & \hbox{$s=i$.}
     \end{array}
     \right.
\end{equation*}
Given $f\in \zD^n(\uk(1,i), \ur)$ let $j_f$ and $i_{\ur}$ be the corresponding coproduct injections as in a). Let
\begin{equation*}
    \pt_{ij}:(TUH)_{\uk(1,i)}\rw (TUH)_{\uk(0,i)}
\end{equation*}
be the functors determined by
\begin{equation*}
    \pt_{ij}i_{\ur}j_f = i_{\ur} j_{f\zd_{ij}}\;.
\end{equation*}
From a), we have
\begin{align*}
   & h_{\uk(0,i)}\,\pt_{ij}\, i_{\ur}\,j_f = h_{\uk(0,i)}\, i_{\ur} \, j_{f\zd_{ij}}= H(f\zd_{ij})\\
& d_{ij}\,h_{\uk(1,i)}\, i_{\ur}\, j_{f} = H(\zd_{ij})\, H(f)\;.
\end{align*}
Since $H\in\Ps\funcat{n}{\Cat}$ and $H_{\uk(0,i)}$ is discrete, it is $H(f\zd_{ij})=H(\zd_{ij})H(f)$ so that, from above,
\begin{equation*}
    h_{\uk(0,i)}\pt_{ij}i_{\ur}j_f = d_{ij} h_{\uk(1,i)}i_{\ur} j_f
\end{equation*}
for each $\ur,f$. We conclude that
\begin{equation*}
     h_{\uk(0,i)}\pt_{ij}=d_{ij} h_{\uk(1,i)}\;.
\end{equation*}
That is, diagram \eqref{fig.pro.strict-mon} commutes.\bk

c) Since, for each $k_i\geq 2$
\begin{equation*}
    [k_i] = [1]\uset{[0]}{\coprod}\; \oset{k_i}{\ldots}\;
    \uset{[0]}{\coprod}[1]
\end{equation*}
we have, for each $\uk = (k_1,...,k_n)\in\zD^n$ with $k_i\geq 2$ for all $1\leq i\leq n$
\begin{equation*}
    \uk = \uk(1,i)\uset{\uk(0,i)}{\coprod}\; \oset{k_i}{\ldots}\;
    \uset{\uk(0,i)}{\coprod}\uk(1,i)\;.
\end{equation*}
Therefore we have a bijection
\begin{equation*}
    \zD^n(\uk,\ur)=\pro{\zD^n\big(\uk(1,i),\ur\big)}{\zD^n(\uk(0,i),\ur)}{k_i}\;.
\end{equation*}
From the proof of b), the functors $\pt_{ij}:(TUH)_{\uk(1,i)}\rw (TUH)_{\uk(0,i)}$ for $j=0,1$ are determined by the functors
\begin{equation*}
    (\overline{\zd_{ij}}, id):\zD^n\big(\uk(1,i),\ur\big)\times H_{\ur} \rw \zD^n\big(\uk(0,i),\ur\big)\times H_{\ur}
\end{equation*}
where $\overline{\zd_{ij}}(g)=g\zd_{ij}$ for $g\in \zD^n\big(\uk(1,i),\ur\big)$ and
\begin{align*}
    & (TUH)_{\uk(1,i)}=\uset{\ur}{\coprod}\zD^n\big(\uk(1,i),\ur\big)\times H_{\ur}\\
    & (TUH)_{\uk(0,i)}=\uset{\ur}{\coprod}\zD^n\big(\uk(0,i),\ur\big)\times H_{\ur}\;.
\end{align*}
It follows that
\begin{align*}
    & \pro{(TUH)_{\uk(1,i)}}{(TUH)_{\uk(0,i)}}{k_i}=\\
    & \uset{\ur}{\coprod}\{ \pro{\zD^n\big(\uk(1,i),\ur\big)}{\zD^n(\uk(0,i),\ur)}{k_i} \}\times H_{\ur}=\\
    & \uset{\ur}{\coprod}\zD^n(\uk,\ur)\times H_{\ur}=(TUH)_{\uk}\;.
\end{align*}
This proves c). \bk

d) From a), ${h_k\,i_{r}\,j_f}=H(f)$ for $f \in \zD^n(\uk,\ur)$. Let $f$ correspond to $(\zd_1,\ldots,\zd_{ki})$ in the isomorphism
\begin{equation*}
    \zD^n(\uk,\ur)=\pro{\zD^n\big(\uk(1,i),\ur\big)}{\zD^n(\uk(0,i),\ur)}{k_i}\;.
\end{equation*}
Then $j_f=(j_{\zd_1},\ldots,j_{\zd_{k_i}})$. Since
\begin{equation*}
    H_{\uk} \cong \pro{H_{\uk(1,i)}}{H_{\uk(0,i)}}{k_i}
\end{equation*}
then $H(f)$ corresponds to $(H(\zd_1),\ldots,H(\zd_{k_i}))$ with $p_i H(f)=H(\zd_i)$. Then for all $f$ we have
\begin{align*}
    h_{\uk}\,i_{\ur}\,j_f &=(H(\zd_1),\ldots,H(\zd_{k_i}))=(h_{\uk(1,i)}\,i_{\ur}\,j_{\zd_1},\ldots, h_{\uk(1,i)}\,i_{\ur}\,j_{\zd_{k_i}})=\\
    & =(h_{\uk(1,i)},\ldots,h_{\uk(1,i)})\,i_{\ur}(j_{\zd_1},\ldots, j_{\zd_{k_i}})=(h_{\uk(1,i)},\ldots,h_{\uk(1,i)})\,i_{\ur}\,j_f\;.\\
\end{align*}
It follows that $h_{\uk}=(h_{\uk(1,i)},\ldots,h_{\uk(1,i)})$.
\end{proof}
\begin{theorem}\label{the-strict-funct}
    The strictification functor
    \begin{equation*}
        \St  : \psc{n-1}{\Cat}\rw \funcat{n-1}{\Cat}
    \end{equation*}
    restricts to a functor
    \begin{equation*}
        L_n: \segpsc{n-1}{\Cat}\rw J_n\catwg{n}
    \end{equation*}
    where $\Nb{n}\catwg{n}$ denotes the image of the multinerve functor $J_n:\catwg{n}\rw\funcat{n-1}{\Cat}$\;.
    Further, for each $H\in \segpsc{n-1}{\Cat}$ and $\uk\in\dop{n-1}$, the map $(L_n H)_{\uk}\rw H_{\uk}$ is an equivalence of categories.
\end{theorem}
\begin{proof}
Let $h:TUH\rw UH$ be as in Section \ref{sbs-pseudo-functors}. As recalled there, to construct the strictification $L=\St H$ of a pseudo-functor $H$ we need to factorize $h=gv$ in such a way that for each $\uk \in \dop{n-1}$, $h_{\uk}$ factorizes as
\begin{equation*}
    (TUH)_{\uk}\xrw{v_{\uk}} L_{\uk} \xrw{g_{\uk}} H_{\uk}
\end{equation*}
with $v_{\uk}$ bijective on objects and $g_{\uk}$ fully faithful. As explained in \cite{PW}, $g_{\uk}$ is in fact an equivalence of categories.

Since the bijective on objects and fully faithful functors form a factorization system in $\Cat$, the commutativity of \eqref{fig.pro.strict-mon} implies that there are functors
\begin{equation*}
    \tilde d_{ij}:L_{\uk(1,i)}\rightrightarrows L_{\uk(0,i)}\qquad j=0,1
\end{equation*}
such that the following diagram commutes:
\begin{equation*}
\xymatrix{
(TUH)_{\uk(1,i)} \ar^{v_{\uk(1,i)}} [rr] \ar_{\pt_{i0}}[d]<-2ex>\ar^{\pt_{i1}}[d]&& L_{\uk(1,i)} \ar^{g_{\uk(1,i)}} [rr] \ar_{\tilde d_{i0}}[d]<-2ex>\ar^{\tilde d_{i1}}[d] && H_{\uk(1,i)}\ar_{d_{i0}}[d]<-2ex>\ar^{d_{i1}}[d]\\
(TUH)_{\uk(0,i)} \ar_{v_{\uk(0,i)}} [rr] && L_{\uk(0,i)} \ar_{g_{\uk(0,i)}} [rr] && H_{\uk(0,i)}\;.
}
\end{equation*}
By Proposition \ref{lem-strict-mon-wg-seg-ps-fun}, $h_{\uk}$ factorizes as
\begin{align*}
    (TUH)_{\uk}=\; & \pro{(TUH)_{\uk(1,i)}}{(TUH)_{\uk(0,i)}}{k_i}\rw\\
    &\xrw{(v_{\uk(1,i)},\ldots,v_{\uk(1,i)})}\pro{L_{\uk(1,i)}}{L_{\uk(0,i)}}{k_i}\rw\\
    & \xrw{(g_{\uk(1,i)},\ldots,g_{\uk(1,i)})}\pro{H_{\uk(1,i)}}{H_{\uk(0,i)}}{k_i}\cong H_{\uk}\;.
\end{align*}
Since $v_{\uk(1,i)}$ and $v_{\uk(0,i)}$ are bijective on objects, so is $(v_{\uk(1,i)},\ldots,v_{\uk(1,i)})$. Since  $g_{\uk(1,i)}$, $g_{\uk(0,i)}$ are fully faithful, so is $(g_{\uk(1,i)},\ldots,g_{\uk(1,i)})$. Therefore the above is the factorization of $h_{\uk}$ and we conclude that
\begin{equation*}
    L_{\uk}\cong\pro{L_{\uk(1,i)}}{L_{\uk(0,i)}}{k_i}\;.
\end{equation*}
Since $L\in\funcat{n-1}{\Cat}$ by Lemma \ref{lem-char-obj} this implies that $L\in\cat{n}$. By \cite{PW}, $L_{\uk}\simeq H_{\uk}$ for all $\uk\in\dop{n-1}$. Therefore, by Lemma \ref{pro-transf-wg-struc}, $L\in\catwg{n}$.
\end{proof}
%%

%%%%%%%%%%%%%%%%%%%%%%%%%%%%%%%%%%%%%%%%%%%%%%%%%%%%%%%%%%%%%%%%%%%%%%%%%


\begin{thebibliography}{xxxx}

\bibitem[1]{BD2} J.Baez, J.Dolan, Higher-dimensional algebra III: $n$-categories and the algebra of opetopes, \emph{Advances in Mathematics} 135(2) (1998) 145-206.

\bibitem[2]{BD} J. C.  Baez, J. Dolan, Higher dimensional algebra and topological quantum field theory, \emph{J. Math. Phys.}, 36 (1995) 6073-6105.

\bibitem[3]{Bk1} C.  Barwick, D. M. Kan, Relative categories: another model for the homotopy theory of homotopy theories. Indag. Math. (N.S.) 23 (2012), no. 1-2, 42-68

\bibitem[4]{B}  M.  A.  Batanin,  Monoidal  globular  categories  as a natural environment for the theory of weak n-categories. \emph{Adv. Math.} 136 (1998), no. 1, 39-103

\bibitem [5]{Ben} J.  B\'{e}nabou, \emph{Introduction to bicategories}, Reports of the Midwest Category Seminar, Springer 1967.

\bibitem[6]{BeRe} J.  Bergner, C. Rezk, Comparison of models for $(\infty,n)$-categories 1, \emph{Geom. Topol.} 17 (2013) 4, 2163-2202.

\bibitem[7]{Be3}  J.   Bergner, Models for $(\infty,n)$-categories and the cobordism hypothesis , \emph{Proc. Sympos. Pure Math.}, 83, Amer. Math. Soc., Providence, RI, 2011.

\bibitem[8]{Be2}  J.  Bergner, Three models for the homotopy theory of homotopy theories, \emph{Topology} 46 (2007), 397-436.

\bibitem[9]{BP} D.  Blanc, S. Paoli, Segal-type algebraic models of n-types, \emph{Algebraic and Geometric Topology} 14 (2014) 3419-3491.

\bibitem[10]{Borc}  F.  Borceux, \emph{Handbook of categorical algebra}, Encyc. Math \& its Appl. 51, Cambridge Univ. Press, 1994

\bibitem[11]{BHS} R.  Brown, P. J. Higgins, R. Sivera, Non abelian algebraic topology, \emph{Eur. Math. Soc. Tracks in Math.} 15, Zurich, (2011).

\bibitem[12]{Cheng2} E.Cheng, Comparing operadic theories of $n$-category, \emph{Homology homotopy and applications} 13(2) 217-249, 2011.

\bibitem[13]{Cheng1} E.Cheng, Weak $n$-categories: opetopic and multitopic foundations, \emph{Journal of Pure and Appl. Algebra} 186(2) 109-137, 2004.

\bibitem[14]{Jard}   P. G.  Goerss, J. F. Jardine, \emph{Simplicial homotopy theory}, Birkhauser  2009.

\bibitem [15]{GPS} R.  Gordon, A. J. Power, R. Street, Coherence for tricategories, \emph{Mem. AMS} 117 no. 558 (1995).

\bibitem[16]{Gr} John W.  Gray, \emph{Formal category theory: adjointes for 2-categories}, Lecture notes in mathematics, Vol. 391, Springer 1974.

 \bibitem[17]{Hai} P.J. Haine, Lifting enhanced factorization systems to functor 2-categories, arXiv:1604.06812v1.

\bibitem[18]{Jo} A.  Joyal, Quasi-categories and kan complexes, \emph{J. Pure Appl. Algebra} 175 (2002), no 1-3, 207-222.

\bibitem [19]{Lack} S.  Lack, Codescent objects and coherence, \emph{J. Pure Applied Algebra}, 175 (2002), 223-241.

\bibitem[20]{Lu2} T.  Leinster,  Higher operads, higher categories. London Mathematical Society Lecture Note Series, 298. \emph{Cambridge University Press, Cambridge}.

\bibitem[21]{L1} T.  Leinster, A survey of definition of n-category, \emph{Theory \& Appl. of  Categories}, 10 (1), (2002), 1-70.

\bibitem[22]{L2} J.  Lurie, \emph{Higher topos theory}, Annals of Mathematics, 170, Princeton University Press, 2009.

\bibitem[23]{PP} S.  Paoli, D. Pronk, A double categorical model of weak 2-categories, \emph{Theory and Applications of Categories}, vol. \textbf{28}, no. 27, 2013, 933--980.

\bibitem[24]{Pa1} S.  Paoli, Homotopically discrete higher categorical structures, preprint 2016, arXiv:1605.05112

\bibitem[25]{Pa} S.  Paoli, Weakly globular $\rm{cat^n\mi groups}$ and Tamsamani's model, \emph{Adv. in Math.} 222 (2009) 621-727.

\bibitem[26]{Pa4} S.  Paoli, Weakly globular $n$-fold categories as a model of weak $n$-categories, in preparation.

\bibitem[27]{Pa3} S.  Paoli, Weakly globular Tamsamani $n$-categories and their rigidification, in preparation.

\bibitem [28]{PW} A.J.  Power, A general coherence result, \emph{J.Pure and Applied Algebra} \textbf{57}(1989), no. 2, pp.~165--173.

\bibitem [29]{Re1} C.  Rezk, A cartesian presentation of weak $n$-categories, \emph{Geom. Topol.} 14 (2010) no 1, 521-571.

\bibitem [30]{Re2} C.   Rezk, A model for the homotopy theory of homotopy theory, \emph{Trans. Amer. Math. Soc.} 353 (2001) no. 3, 973-1007.

\bibitem [31]{Simp} C.    Simpson, \emph{Homotopy theory of higher categories}, Cambridge University Press, Cambridge 2012.

\bibitem [32]{S2}  C.   Simpson, Homotopy types of strict 3-groupoids , preprint, 1988 math.CT/9810059.

\bibitem [33]{Str} R. Street, Two constructions of lax functors, \emph{Cahiers Topol. Geom. Differentielles} 13 (1972), 217-264.

\bibitem[34]{Ta} Z.  Tamsamani, Sur des notions de $n$-cat\'{e}gorie et $n$-groupoide non-strictes via des ensembles multi-simpliciaux, \emph{K-theory}, \textbf{16}, (1999), 51-99.

\bibitem[35]{Thomas} R. W. Thomason, Homotopy colimits in the category of small categories, \emph{Math. Proc. camb. Phil. Soc.} (1979), 85, 91.

\bibitem[36]{Ve} D.  Verity, Weak complicial sets II: Nerves of complicial Grey categories, \emph{Comtemp. Math.} 431, 2007.


\end{thebibliography}
\end{document}